\documentclass[11pt]{amsart}
\usepackage{graphicx}
\usepackage{latexsym}
\usepackage{amsfonts,amsmath,amssymb}
\newtheorem{theorem}{Theorem}[section]

\newtheorem{lemma}[theorem]{Lemma}

\usepackage{color}

\def\eps{\varepsilon}
\def\la{\lambda}
\def\a{\alpha}
\def\be{\beta}

\def\ga{\gamma}
\def\part{\partial}

\def\b1{\bold 1}

\newcommand{\beq}{\begin{equation}}
\newcommand{\eeq}{\end{equation}}

\theoremstyle{remark}

\numberwithin{equation}{section}
\date{\today}

\begin{document}

\title[Perfect partitions]{Perfect partitions of a random set of integers}  
\author{Boris Pittel}
\address{Department of Mathematics, The Ohio State University, $231$ West $18$-th Avenue, Columbus, Ohio $43210-1175$ (USA)}
\email{bgp@math.ohio-state.edu}


\begin{abstract} Let $X_1,\dots, X_n$ be independent integers distributed uniformly on $\{1,\dots, M\}$, $M=M(n)\to\infty$ however slow. A partition $S$ of $[n]$ into
$\nu$ non-empty subsets $S_1,\dots, S_{\nu}$ is called perfect, if all $\nu$ values $\sum_{j\in S_{\a}}X_j$ are equal. For a perfect partition to exist, $\sum_j X_j$ has to be divisible by $\nu$.
For $\nu=2$, Borgs et al. proved, among other results, that, conditioned on $\sum_j X_j$ being even, with high  probability a perfect partition exists if 
$\kappa:=\lim \tfrac{n}{\log M}>\tfrac{1}{\log 2}$, and that w.h.p. no perfect partition exists if $\kappa<\tfrac{1}{\log 2}$. We prove that w.h.p. no perfect partition exists if $\nu\ge 3$ and  $\kappa<\tfrac{2}{\log \nu}$. We identify the range of $\kappa$ in which the expected number of perfect partitions is exponentially high. We show that for $\kappa> \tfrac{2(\nu-1)}{\log[(1-2\nu^{-2})^{-1}]}$ the total number of perfect partitions is exponentially high with probability $\gtrsim (1+\nu^2)^{-1}$.
\end{abstract}

\dedicatory{In loving memory of Irina Pittel, my soulmate for sixty years.}

\maketitle

\section{Introduction and main result} To liberally cite from Wikipedia \cite{Par} the partition problem is to determine if a given multiset of integers can be partitioned into two subsets such that the sums of numbers in the subsets are the same. An optimization version of the partition problem is to partition the multiset into two subsets such that the absolute value of the difference between the sums is minimized. In the multiway partition problem, there is an integer $\nu$ and the problem is to determine whether the multiset can be partitioned into $\nu$ subsets, each of the same total sum. We are interested in the case when the integers to be partitioned are random.

Specifically  we have integers $X_1,\dots, X_n$, which are independent, each distributed uniformly on
$[M]:=\{1,\dots, M\}$. Given $ \nu\ge 2$, consider an ordered partition of $[n]$ into non-empty subsets $S_1,\dots, S_{ \nu}$,
which we denote $\bold S$. 
 A partition $\bold S$ of $[n]$ 
 is called {\it perfect\/} if $\sum_{j\in S_{\a}}X_j\equiv \tfrac{1}{\nu}\sum_{j\in [n]}X_j$, $\a\in [\nu]$. For a perfect partition to exist, it is necessary that $\sum_{j\in [n]}X_j$ is divisible by $\nu$. 
 For $ \nu=2$, Borgs, Chayes, and Pittel \cite{BCP} proved the following result, among a series of other, sharper claims. Conditioned on $\sum_j X_j$ even, with high probability, a perfect partition exists (does not exist) if $\lim\tfrac{n}{\log M}>\tfrac{1}{\log 2}$ (if $\lim\tfrac{n}{\log M}<\tfrac{1}{\log 2}$). In a follow-up paper Borgs, Chayes, Mertens, and Pittel \cite{BCMP} analyzed a constrained version, when the cardinalities of the two sets are given. A constrained version with $k$ sets of the same cardinality was analyzed later by Bovier and Kurkova \cite{Bov}. We refer the reader to Graham \cite{Gra}, Gent and Walsh \cite{Gen},  Hayes \cite{Hay}, Karmarkar, Karp, Lueker, and Odlyzko \cite{Kar}, Mertens \cite{Mer}, and 
 Yakir \cite{Yak}, 
for the long history of the problem and its various aspects, from computational issues to conceptual connections with statistical physics.

In 2021 Tim Varghese and George Varghese \cite{Var} posed a problem of extending probabilistic analysis of perfect partitions into two subsets to multiway partitions, i.e. partitions of a set of random integers into $\nu>2$ subsets with the same sum of integers in each subset. Here is what we have found.
\begin{theorem}\label{main} Let $\nu\ge 3$. Suppose $M=M(n)\to\infty$, $M=exp(O(n))$. {\bf (1)\/} 
Expected number of perfect partitions tends to zero, if $\lim \tfrac{n}{\log M}<\tfrac{2}{\log \nu}$. {\bf (2)\/} Expected number of perfect partitions is exponentially large 
{\bf(a)\/} for $\nu=3$ and $\lim \tfrac{n}{\log M}\in \bigl(\tfrac{2}{\log 3},\infty\bigr)\setminus\{2,4\}$,  
and {\bf (b)\/} for $\nu\ge 4$ and $\lim \tfrac{n}{\log M}\in \bigl(\tfrac{2}{\eta(\nu)},\infty\bigr)\setminus\{2(\nu-1)\}$, $\eta(\nu)$ being a single root of $\tfrac{e^{(\nu-1)\eta}}{\nu e\eta}-1$ in $\bigl(\tfrac{1}{\nu-1}, \tfrac{2\log\nu}{\nu-1}\bigr)$.
\end{theorem}
\noindent We conjecture that $\lim \tfrac{n}{\log M}=2(\nu-1)$ for $\nu \ge 3$, and $\lim \tfrac{n}{\log M}=2$ for $\nu=3$ need not be excluded for {\bf (2)\/} to hold. What happens when $\nu>3$ and $\lim\tfrac{n}{\log M}\in \bigl(\tfrac{2}{\log\nu}, \tfrac{2}{\eta(\nu)}\bigr)$? 
Our second result is
\begin{theorem}\label{second} Suppose $\nu\ge 3$. If $\lim\tfrac{n}{\log M}>\tfrac{2(\nu-1)}{\log[(1-2\nu^{-2})^{-1}]}$, then with the
limiting probability $\gtrsim (1+\nu^2)^{-1}$ the total number of perfect partitions is exponentially large, of the same order of magnitude as the expected number of perfect partitions. 
\end{theorem}
\noindent The lower bound $(1+\nu^2)^{-1}$ is below $\nu^{-1}\!=\!\lim_{M\to\infty} \Bbb P\bigl(\sum_{j\in [n]} X_j\!\equiv\! 0(\text{mod }\nu)\bigr)$, uniformly for $n\ge 1$. (A simple, elegant proof of the latter was conveyed to me by Huseyin Acan \cite{Aca}.)
\section{Integral formulas}
\begin{lemma}\label{newlem1} {Suppose $\nu>2$. Let $Z_n$ denote the random total number of all perfect ordered partitions of $[n]$ into $ \nu$ non-empty
subsets. {\bf (1)\/} Then 
\begin{equation}\label{new1}
\begin{aligned}
&\,Z_n=\frac{1}{(2\pi)^{ \nu}}\!\!\!\int\limits_{\bold x\in [-\pi,\pi]^{ \nu}}\prod_{j\in [n]}f(\bold x,X_j)\,\,d\bold x, \quad \bold x:=\{x_1,\dots,x_{\nu}\},\\
&f(\bold x,\eta):=\sum_{\a\in [\nu]}\exp\bigl(iy_{\a}\eta\bigr),\quad y_{\a}=y_{\a}(\bold x):=\nu x_{\a}-\sum_{\be\in [\nu]}x_{\be}.\\
\end{aligned}
\end{equation}
Consequently, denoting by $X$ a random variable distributed as each of $X_j$, we have
\begin{equation}\label{new1.5}
\Bbb E\bigl[Z_n\bigr]
=\!\frac{1}{(2\pi)^{\nu}}\int_{\bold x\in [-\pi,\pi]^{\nu}}\Bbb E^n\bigl[
f(\bold x,X) \bigr]d\bold x.
\end{equation}
{\bf (2)\/} Furthermore
\begin{equation}\label{new1.51}
\begin{aligned}
\Bbb E[Z_n^2]&=\frac{1}{(2\pi)^{ 2\nu}}\!\!\!\int\limits_{\bold x, \bold x'}\Bbb E^n\bigl[f(\bold x,X)
\overline{f(\bold x',X)}\bigr]\,\,d\bold x d\bold x'.
\end{aligned}
\end{equation}
}
\end{lemma}
\begin{proof}{ {\bf (1)\/}
Given an ordered partition $\bold S=\{S_1,\dots,S_{ \nu}\}$ of $[n]$ into non-empty subsets $S_{\a}$, define $Y_{\a}:=\sum_{j\in S_{\a}}X_j$,  so that $\sum_{\be\in [\nu]}Y_{\be}=\sum_{j\in [n]}X_j$.  $\bold S$  is perfect if and only if
\begin{equation}\label{new2.5}
 \nu \cdot Y_{\a}  -\sum_{\be\in [ \nu]}Y_{\be}=0,\quad \a\in [ \nu].
\end{equation}
 Clearly, for a perfect partition to exist, 
 $\sum_{j\in [n]}X_j$ must be divisible by $\nu$.
Let $\Bbb I(A)$ denote the indicator of an event $A$. Then
\begin{multline*}
\Bbb I(\bold S\text{ perfect})=\prod_{\a\in [ \nu]}\Bbb I\biggl( \nu\cdot Y_{\a}-\sum_{\be\in [ \nu]}Y_{\be}=0\biggr)\\
=\prod_{\a\in [ \nu]}\frac{1}{2\pi}\int_{x\in [-\pi,\pi]}\exp\biggl[ix\biggl( \nu \cdot Y_{\a}  -\sum_{\be\in [ \nu]}Y_{\be}\biggr)\biggr]\,dx\\
=\frac{1}{(2\pi)^{ \nu}}\int_{\bold x\in [-\pi,\pi]^{ \nu}}\exp\biggl[i\sum_{\a\in  [\nu]}x_{\a}\biggl( \nu\cdot Y_{\a}-\sum_{\be\in [ \nu]}Y_{\be}\biggr)
\biggr]\,d\bold x,\quad \bold x:=\{x_{\a}\}_{\a\in [\nu]}.
\end{multline*}
Summing the above equation over all ordered partitions $\bold S$ with nonempty $S_{\a}$'s, and interchanging summation and integration on RHS, we obtain
\begin{equation}\label{new3}
Z_n=\frac{1}{(2\pi)^{ \nu}}\int_{\bold x\in [-\pi,\pi]^{ \nu}}\sum_{\bold S}\exp\biggl[i\sum_{\a\in [ \nu]}x_{\a}\biggl( \nu\cdot Y_{\a}-\sum_{\be\in [ \nu]}Y_{\be}\biggr)
\biggr]\,d\bold x.
\end{equation}
What if $\bold S$ is such that some $S_{\a_0}=\emptyset$? In this case $Y_{\a_0}=\sum_{j\in S_{\a_0}}X_j=0$; so the factor by $ix_{\a_0}$ in the exponent sum is $-\sum_{j\in [n]}X_j\neq 0$. Then the integral of the corresponding exponential 
function over the (hyper)cube $[-\pi,\pi]^{\nu}$ is zero. 
Therefore we may, and will  
extend the sum in \eqref{new3} to all ordered sequences $\bold S$ of sets $S_1,S_2,\dots,S_{ \nu}$, {\it empty or non-empty\/}, that form a partition of $[n]$. Regrouping summands in the exponent we obtain 
\begin{align*}
&\exp\biggl[i\sum_{\a\in [ \nu]}x_{\a}\biggl( \nu\cdot Y_{\a}-\sum_{\be\in [ \nu]}Y_{\be}\biggr)
\biggr]=\prod_{j\in [n]}\exp(i\sigma_jX_j),\\
&\quad\,\,\sigma_j=\sum_{\a\in [ \nu]}x_{\a}\Bigl(\nu\,\delta_{\a}(j)-1\Bigr), \quad \delta_{\a}(j):= \Bbb I(j\in S_{\a}).
\end{align*}
Such a {\it relaxed\/} partition $\bold S$ is uniquely determined by values of $\delta_{\a}(j)\in\{0,1\}$, subject  to $\sum_{\a\in [\nu]}\delta_{\a}(j)=1$, ($j\in [n]$), but not to $\sum_{j\in [n]}\delta_{\a}(j)\ge 1$, ($\a\in [\nu]$), because empty $S_{\a}$ are allowed now. Thus, $\delta_{\a}(j)$
vary {\it independently\/} for different $j$.  So, introducing
\[
\Sigma(\bold x):=\biggl\{\sum_{\a\in [ \nu]}x_{\a}\bigl(\nu\,\delta_{\a}-1\bigr): \delta_{\ga}\in \{0,1\},\,
\sum_{\ga\in [\nu]}\delta_{\ga}=1\biggr\},
\]
for the sum over relaxed partitions we obtain
\begin{multline*}
\sum_{\bold S}\exp\biggl[i\sum_{\a\in [\nu]}x_{\a}\biggl( \nu\cdot Y_{\a}-\sum_{\be\in [ \nu]}Y_{\be}\biggr)
\biggr]=\prod_{j\in [n]}\sum_{\sigma\in \Sigma(\bold x)}\exp(iX_j \sigma)\\
=\prod_{j\in [n]}\sum_{\a\in [\nu]}\exp\biggl[i\biggl(\nu x_{\a}-\sum_{\be\in [\nu]}x_{\be}\biggr)X_j\biggr]
=\prod_{j\in [n]}\sum_{\a\in [\nu]}\exp\bigl(iy_{\a}(\bold x)X_j\bigr).
\end{multline*}
We conclude that
\begin{equation*}
Z_n=\frac{1}{(2\pi)^{ \nu}}\int\limits_{\bold x\in [-\pi,\pi]^{ \nu}}
\prod_{j\in [n]}\sum_{\a\in [\nu]}\exp\bigl(iy_{\a}(\bold x)X_j\bigr)\,\,d\bold x,
\end{equation*}
which is \eqref{new1} in the lemma. Equating the expectations of both sides and using independence of $X_j$'s, we get \eqref{new1.5}. 

{\bf (2)\/}  The equation \eqref{new1} implies that
\[
Z_n^2=\frac{1}{(2\pi)^{ 2\nu}}\!\!\!\int\limits_{\bold x, \bold x'}\prod_{j\in [n]}f(\bold x,X_j)
\overline{f(\bold x',X_j)}\,\,d\bold x d\bold x',
\]
so that
\[
\Bbb E[Z_n^2]=\frac{1}{(2\pi)^{ 2\nu}}\!\!\!\int\limits_{\bold x, \bold x'}\Bbb E^n\bigl[f(\bold x,X)
\overline{f(\bold x',X)}\bigr]\,\,d\bold x d\bold x'.
\]

}
\end{proof}

\section{Expected number of perfect partitions}
\noindent   
Using Lemma \ref{newlem1}, we prove

\begin{theorem}\label{newthm1} {Let $\nu\ge 3$. Suppose that $M=M(n)\to\infty$ however slowly.  If $\lim\tfrac{n}{\log M}<
\tfrac{2}{\log \nu}$, then $\lim\Bbb E[Z_n]=0$. {\bf (i)\/} Suppose $\nu=3$. If $\lim\tfrac{n}{\log M}\in \bigl(\tfrac{2}{\log 3},\infty\bigr)\setminus\{2,4\}$, then
\begin{equation}\label{?}
\Bbb E\bigl[Z_n\bigr]=(1+o(1))\frac{\nu^n}{M^{\nu-1}}\cdot\frac{\nu^{\nu-3/2}}{(2\pi\nu c_Mn)^{\frac{\nu-1}{2}}}\to\infty,\\
\end{equation}
$c_{\scriptscriptstyle M}:=M^{-2}\Bbb E[X^2]= 1/3+O(M^{-1})$. 
{\bf (ii)\/} Suppose $\nu\ge 4$. Then \eqref{?} holds if $\lim\tfrac{n}{\log M}\in \bigl(\tfrac{2}{\eta(\nu)},\infty)\setminus
\{2(\nu-1)\}$; here $\eta(\nu)\in \bigl(\tfrac{1}{\nu-1},\tfrac{2\log \nu}{\nu-1}\bigr)$ 
is the larger of two roots of $\tfrac{e^{(\nu-1)\eta}}{\nu e \eta}-1$, the smaller one being below $\tfrac{1}{\nu-1}$.}
\end{theorem}
\noindent {\bf Note.\/} 
It is unclear whether $\Bbb E[Z_n]\to\infty$ if 
$\nu\ge 3$ and $\lim\tfrac{n}{\log M}=2(\nu-1)$, or if $\nu=3$ and  $\lim\tfrac{n}{\log M}=2$. Perhaps $\Bbb E[Z_n]\to\infty$ in all these cases, albeit at a rate lower then for $\lim\tfrac{n}{\log M}$ arbitrarily close to, but different from 
$2(\nu-1)$ for $\nu\ge 3$, and from $2$ for $\nu=3$, respectively. The case
 $\nu\ge 4$, $\lim\tfrac{n}{\log M}\in \bigl(\tfrac{2}{\log\nu}, \tfrac{2}{\eta(\nu)}\bigr)$ remains open as well.

\begin{proof}  By \eqref{new1.5}, we have
\begin{align*}
\Bbb E\bigl[Z_n\bigr]
&=\!\frac{1}{(2\pi)^{\nu}}\!\!\!\int\limits_{\bold x\in [-\pi,\pi]^{\nu}}\!\!\!\Bbb E^n\bigl[
f(\bold x,X) \bigr] d\bold x,\\
\Bbb E\bigl[f(\bold x,X)\bigr]&=\sum_{\a\in [\nu]}\Bbb E\bigl[e^{iy_{\a}X}\bigr],\quad y_{\a}=y_{\a}(\bold x):=\nu x_{\a}-\sum_{\be\in [\nu]}x_{\be}.
\end{align*}
{\bf (I)\/} Since $X$ is integer-valued, $\Big|\Bbb E\bigl[f(\bold x,X) \bigr] \Big|$ 
certainly attains its  maximum $\nu$ at every $\bold x$ such that for each $\a\in [\nu]$ we have
$y_{\a}=y_{\a}(\bold x)=k_{\a}\pi$, and $k_{\a}$ is even. Here necessarily $\sum_{\a\in [\nu]}k_{\a}=0$, since $\sum_{\a\in [\nu]}y_{a}(\bold x)=0$. Given such $\bold k=\{k_{\a}\}_{\a\in [\nu]}$, the set of $\bold x$ satisfying
\begin{equation}\label{3.01}
\nu x_{\a}-\sum_{\be\in [\nu]} x_{\be}= k_{\a}\pi,\quad (\a\in [\nu]),
\end{equation}
is a straight line $\mathcal L_{\bold k}$ in $\Bbb R^{\nu}$, since the $\nu\times\nu$ matrix on the LHS has rank $\nu-1$. We will prove that $\Big|\Bbb E\bigl[f(\bold x,X) \bigr] \Big|<\nu$ everywhere else.
Let us enumerate the lines $\mathcal L_{\bold k}$ that contain {\it interior points\/} of  the cube $[-\pi,\pi]^{\nu}$.
To this end notice that each line's parametric equation is $x_{\a}=\nu^{-1}(t+k_{\a}\pi)$, $t\in \Bbb R$. This line contains an interior point $\bold x$ of the cube $[-\pi,\pi]^{\nu}$ if and only if for some $t$
\[
\nu^{-1}(t+k_{\a}\pi)\in (-\pi,\pi), \quad\forall\,\a\in [\nu],
\]
or equivalently
\begin{equation}\label{3.015}
-\pi(\nu+\min k_{\a})<t<\pi(\nu-\max k_{\a}).
\end{equation}
In particular, $\max k_{\a}-\min k_{\a}\le 2(\nu-1)$, since $k_{\a}$ are even.
For $\nu=2$ we have $k_1=k_2=0$ since $k_1+k_2=0$. From now on we focus on $\nu>2$. We consider these (admissible) lines only, since any other line can share at most one point with the cube.
And the set of all such touch points is finite.

Indeed, suppose $\mathcal L$ does not contain interior points of the cube. If $\mathcal L$'s segment is in the cube, then this segment, i.e. a non-zero multiple of $\bold e:=(1,\dots,1)\in E^{\nu}$, belongs to one of $2\nu$ faces, cubes of dimension $\nu-1$. Hence the segment is orthogonal to a coordinate vector $(0,\dots,0,1,0,\dots 0)\in E^{\nu}$-- contradiction. So $\mathcal L$ shares at most one point with the cube. Any such touch point $\bold x$ satisfies $y_{\a}(\bold x)=\pi k_{\a}$, $(\a\in [\nu])$, and it has to have a component $x_{\a'}=\pi$, and a component $x_{\a''}=-\pi$, for that matter. Otherwise $\mathcal L$ would have contained interior points of the cube. The linear system $y_a(\bold x)=k_{\a}\pi$, $x_{\a'}=\pi$ has at most one solution in the cube, because the associated matrix has rank $\nu$. So, since $|k_{\a}|\le 2(\nu-1)$, $\a\in [\nu]$, the number of such solutions (i.e. touch points) is at most $\nu(2\nu-1)^{\nu}$.

By \eqref{3.015}, $\bold k=\{k_{\a}\}_{\a\in [\nu]}$, which consists of even numbers adding up to $0$, is admissible if and only if 
$
d(\bold k):=2\nu+\min_{\a}k_{\a}-\max_{\a}k_a\ge 2, 
$
since the range of the parameter $t$ is $\pi d(\bold k)$, and 
the length of the corresponding line segment is 
\begin{equation}\label{3.016}
L(\bold k)=\pi d(\bold k)\nu^{-1/2}=\pi \nu^{-1/2} \bigl(2\nu+\min_{\a}k_{\a}-\max_{\a}k_a\bigr).
\end{equation}
Let even numbers $a\le 0\le b$ be generic values of $\min_{\a}k_{\a}$ and $\max_{\a}k_{\a}$. 
Let $\mathcal N(a,b)$ be the total number of admissible $\{k_{\a}\}_{\a\in [\nu]}$ with parameters $a$, $b$. Obviously $\mathcal N(a,a)=\Bbb I(a=0)$, and $\mathcal N(a,a+2)=0$. Consider $r:=b-a\in [4,2(\nu-1)]$, $2(\nu-1)$  coming from $d(\bold k)\ge 2$.  A generic $\{k_{\a}\}$ with given $a=
\min_{\a}k_{\a}$,  $b=\max_{\a}k_{\a}$, contains some $\mu_1>0$ components equal $a$, and $\mu_2>0$ components 
equal $b$, with $\mu_1+\mu_2\le \nu$, and $\mu_3:=\nu-\mu_1-\mu_2$ remaining even components, with values strictly between $a$ and $b$, such that the total sum of all $\nu$ components is zero. Consequently, if $b\ge a+4$ then
\begin{equation}\label{3.02}
\begin{aligned}
N(a,b)&=[\zeta^0]\sum_{\mu_1,\,\mu_2>0}\frac{\nu!}{\mu_1!\,\mu_2!\,\mu_3!}\cdot \zeta^{\mu_1 a+\mu_2 b}\biggl(\sum_{\text{even }j=a+2}^{b-2}\zeta^j\biggr)^{\mu_3},\\
\end{aligned}
\end{equation}
with the last sum equal 
\begin{equation}\label{3.025}
\zeta^{a+2}\sum_{\text{even }j=0}^{r-4}\zeta^j=\tfrac{\zeta^{a+2}(1-\zeta^{r-2})}{1-\zeta^2},\quad r:=b-a.
\end{equation}
So, by inclusion-exclusion, the RHS function of $\zeta$ in \eqref{3.02} equals
\begin{multline*}
\biggl[\zeta^a+\zeta^b+\frac{\zeta^{a+2}(1-\zeta^{r-2})}{1-\zeta^2}\biggr]^{\nu}-\biggl[\zeta^a+\frac{\zeta^{a+2}(1-\zeta^{r-2})}{1-\zeta^2}\biggr]^{\nu}\\
-\biggl[\zeta^b+\frac{\zeta^{a+2}(1-\zeta^{r-2})}{1-\zeta^2}\biggr]^{\nu}+\biggl[\frac{\zeta^{a+2}(1-\zeta^{r-2})}{1-\zeta^2}\biggr]^{\nu}\\
=\zeta^{a\nu}\biggl[\biggr(\frac{1-\zeta^{r+2}}{1-\zeta^2}\biggr)^{\nu}-\biggr(\frac{1-\zeta^r}{1-\zeta^2}\biggr)^{\nu}
-\biggl(\frac{\zeta^2-\zeta^{r+2}}{1-\zeta^2}\biggr)^{\nu}+\biggl(\frac{\zeta^2-\zeta^r}{1-\zeta^2}\biggr)^{\nu}\biggr].
\end{multline*}
So, 
we have
\begin{equation}\label{3.03}
\begin{aligned}
\mathcal N(a,b)&=[\zeta^{-a\nu}] \mathcal F_r(\zeta),\\
\mathcal F_r(\zeta)&=
\biggr(\frac{1-\zeta^{r+2}}{1-\zeta^2}\biggr)^{\nu}-\biggr(\frac{1-\zeta^r}{1-\zeta^2}\biggr)^{\nu}\\
&\qquad\qquad\qquad-\biggl(\frac{\zeta^2-\zeta^{r+2}}{1-\zeta^2}\biggr)^{\nu}+\biggl(\frac{\zeta^2-\zeta^r}{1-\zeta^2}\biggr)^{\nu}.
\end{aligned}
\end{equation}
Since $b=a+r>0$, we must have $\mathcal N(a,b)=0$ for $a\le -r$. As a partial validation, one can check that the top RHS 
of \eqref{3.03} is indeed zero for $a\le -r$, and also for $a=0$.
What we will actually need is
\begin{equation}\label{3.04}
\mathcal N(r):=\sum_{\text{even }a<0}\mathcal N(a,a+r),\quad \text{even }r\in [4,2(\nu-1)],
\end{equation}
the total number of the relevant $\nu$-tuples of even integers $\{k_{\a}\}_{\a\in [\nu]}$ adding up to $0$, such that $\max_{\a\in [\nu]}k_{\a}-\min_{\a\in [\nu]}k_{\a}=r$. 
In fact, ultimately we are after a single number, namely
\[
\mathcal M=\mathcal M(\nu):=2\nu+\sum_{\text{even }r\in [4,2(\nu-1)]}(2\nu-r)\mathcal N(r),
\]
since $\pi\nu^{-1/2}\mathcal M$ is the total length of the in-cube segments of lines for even tuples 
$\bold k$. Combining \eqref{3.03} and \eqref{3.04} and changing the order of double summation we obtain
\begin{equation}\label{3.05}
\mathcal M
=2\nu+\sum_{\text{even }a\le 0} [\zeta^{-a\nu}]\sum_{\text{even }r\in [4,2(\nu-1)]}(2\nu-r)\mathcal F_r(\zeta).
\end{equation}
We included $a=0$ since $\mathcal F_r(0)=0$. Using the formula for $\mathcal F_r(\zeta)$, and telescoping, we get: with $r$ even,
\begin{multline}\label{3.055}
\sum_{r\in [4,2(\nu-1)]}\!\!\!\!\!\!\!\!\!\mathcal F_r(\zeta)
\!=\biggr(\!\frac{1-\zeta^{2\nu}}{1-\zeta^2}\!\biggr)^{\nu}\!-\biggr(\!\frac{1-\zeta^4}{1-\zeta^2}\!\biggr)^{\nu}
\!-\biggl(\!\frac{\zeta^2-\zeta^{2\nu}}{1-\zeta^2}\!\biggr)^{\nu}\!\!+\biggl(\!\frac{\zeta^2-\zeta^4}{1-\zeta^2}\!\biggr)^{\nu}\\
=\biggr(\!\frac{1-\zeta^{2\nu}}{1-\zeta^2}\!\biggr)^{\nu}-\biggl(\!\frac{\zeta^2-\zeta^{2\nu}}{1-\zeta^2}\!\biggr)^{\nu}-(1+\zeta^2)^{\nu}
+\zeta^{2\nu}:=S_1(\zeta^2).
\end{multline}
Then
\[
\sum_{\text{even }a\le 0}[\zeta^{-a\nu}]\sum_{\text{even }r\in [4,2(\nu-1)]}\mathcal F_r(\zeta)
=\sum_{\a\ge 0}[z^{\a\nu}]S_1(z),
\]
and using a simple formula
\begin{equation}\label{3.057}
\sum_{\a\ge 0}\, [z^{\a\nu}]\,S(z)=\tfrac{1}{\nu}\sum_{j=0}^{\nu-1}S(z_j), \quad z_j:=\exp\bigl(i\tfrac{2\pi j}{\nu}\bigr),
\end{equation}
we obtain
\begin{multline}\label{coeff1}
\sum_{\a\ge 0}[z^{\a\nu}]S_1(z)=\tfrac{1}{\nu}\sum_{j=0}^{\nu-1}\Bigl[\Bigl(\tfrac{1-z_j^{\nu}}{1-z_j}\Bigr)^{\nu}-\Bigl(\tfrac{1-z_j^{\nu-1}}{1-z_j}\Bigr)^{\nu}\Bigr] - 1\\
=\tfrac{1}{\nu}\bigl[\nu^{\nu}-(\nu-1)^{\nu}-(\nu-1)(-1)^{\nu}\bigr]-1.
\end{multline}
Combinatorially, the RHS is the total number of all tuples associated with lines {\it containing interior points\/} of the cube $[-\pi,\pi]^{\nu}$, distinct from the diagonal $\bold x =\ga \bold e$, i.e. $k_{\a}\equiv 0$. In particular, it is $0$ for $\nu=2$, and 
it is $6=3!$, as it should be: the former because the single admissible $2$-tuple is $\{0,0\}$, and the latter because 
the admissible tuples $\{k_1,k_2,k_3\}\neq\{0,0,0\}$ are $3!$ permutations of the single tuple $\{-2,0,2\}$.   

Similarly telescoping again,
\begin{multline}\label{3.056}
\sum_{\text{even }r\in [4,2(\nu-1)]}\!\!\!\!\!\!r\,\mathcal F_r(\zeta)=2\nu\biggr(\!\frac{1-\zeta^{2\nu}}{1-\zeta^2}\!\biggr)^{\nu}\!
-2\nu\zeta^{2\nu}\biggl(\!\frac{1-\zeta^{2(\nu-1)}}{1-\zeta^2}\!\biggr)^{\nu}\\
\quad+2\!\!\!\!\!\sum_{\text{even }r\in [6,2\nu]}\biggl[-\biggl(\frac{1-\zeta^r}{1-\zeta^2}\biggr)^{\nu}+
\zeta^{2\nu}\biggl(\frac{1-\zeta^{r-2}}{1-\zeta^2}\biggr)^{\nu}\biggr]\\
\quad-4(1+\zeta^2)^{\nu}+4\zeta^{2\nu}.
\end{multline}
Now, analogously to \eqref{coeff1}, we will replace $\zeta^2$ with $z$, making $\zeta^{2\nu}=z^{\nu}=1$ for all $z_j=\exp\bigl(i\tfrac{2\pi j}{\nu}\bigr)$. Hence we may, and will replace in \eqref{3.056} the factor $\zeta^{2\nu}$ with $1$ beforehand, and make the middle sum perfectly amenable to another round of telescoping. Diagrammatically,
\begin{align*}
\sum_{\text{even }r\in [4,2(\nu-1)]}\!\!\!\!\!\!r\,\mathcal F_r(\zeta)\Longrightarrow
&S_2(z):=2(\nu-1)\Bigl(\tfrac{1-z^{\nu}}{1-z}\Bigr)^{\nu}\\
&-2\nu\Bigl(\tfrac{1-z^{\nu-1}}{1-z}\Bigr)^{\nu}-2(1+z)^{\nu}+4z^{\nu}.
\end{align*}
And
\begin{multline}\label{coeff2}
\sum_{\text{even }a\le 0}[\zeta^{-a\nu}]\sum_{\text{even }r\in [4,2(\nu-1)]}r\mathcal F_r(\zeta)\\
=\sum_{\a\ge 0}[z^{\a\nu}]S_2(z)=\tfrac{2}{\nu}\sum_{j=0}^{\nu-1}\Bigl[(\nu-1)\Bigl(\tfrac{1-z_j^{\nu}}{1-z_j}\Bigr)^{\nu}
-\nu\Bigl(\tfrac{1-z_j^{\nu-1}}{1-z_j}\Bigr)^{\nu}\Bigr]\\
=2\bigl(1-\tfrac{1}{\nu}\bigr)-2(\nu-1)^{\nu}-2(\nu-1)(-1)^{\nu}.
\end{multline}
Combining \eqref{3.016}, \eqref{coeff1}, and \eqref{coeff2}, we obtain 
\begin{multline}\label{M=}
\mathcal M=2\nu\Bigl(1+\sum_{\text{even }a\le 0}[\zeta^{-a\nu}]\sum_{\text{even }r\in [4,2(\nu-1)]}\mathcal F_r(\zeta)\Bigr)\\
-\sum_{\text{even }a\le 0}[\zeta^{-a\nu}]\sum_{\text{even }r\in [4,2(\nu-1)]}r\mathcal F_r(\zeta)=2\nu^{\nu-1}.
\end{multline}

\noindent Combining \eqref{3.016} and  \eqref{M=}, 
we have proved 
\begin{lemma}\label{lem1.6} The total length of the in-cube segments of lines for even tuples $\bold k$ is $2\pi \nu^{\nu-3/2}$.
\end{lemma}
\noindent The reason that asymptotically only $\pi\nu^{-1/2} \mathcal M$ matters is that the dominant contribution to $\Bbb E[Z_n]$ comes from thin, parallel, cylinders in the cube $[-\pi,\pi]^{\nu}$, each enclosing its own line segment, with the  integrand asymptotically  translation-invariant along the line segment, and dependent only on the Euclidean distance from $\bold x$ to the point on this segment with the same $\sum_{\a\in [\nu]}x_{\a}$.

{\bf (II)\/} {To prove this claim, our first step is to upper bound $\Big|\Bbb E\bigl[f(\bold x,X)\bigr]\Big|$ for $\bold x$'s outside the lines
\[
\mathcal L _{\bold k}=\{\bold x\in [-\pi,\pi]^{\nu}: y_{\a}(\bold x)=k_{\a}\pi\},\quad \sum_{\a}k_{\a}=0,\,\,k_{\a}\text{ are all even}.
\]
(We remind that $y_{\a}(\bold x)=\nu x_{\a}-\sum_{\be}x_{\be}$.) Start with
\begin{align*}
&\Bbb E[f(\bold x,X)]:=\sum_{\a\in [\nu]}\phi(y_{\a}),\quad y_{\a}=y_{\a}(\bold x),\\
\phi(y)&:=\Bbb E[e^{iyX}]=\tfrac{1}{M}\sum_{j\in [M]} e^{iyj}=\frac{e^{iy}(e^{iyM}-1)}{M(e^{iy}-1)},
\end{align*}
so that
\begin{equation}\label{new3.2}
\Big|\Bbb E\bigl[f(\bold x,X)\bigr]\Big|\le\sum_{\a\in [\nu]}|\phi(y_{\a})|,\quad |\phi(y)|=\Big|\frac{\sin (My/2)}{M\sin(y/2)}\Big|.
\end{equation}
To proceed, 
given $|y|\le 2\pi(\nu-1)$, define an even integer $k(y)$ by condition: 
\[
\min_{\text{even }k: |k|\le 2(\nu-1)} |y-k\pi|=|y-k(y)\pi|. 
\]
So, $k(y)\pi$ is an even multiple of $\pi$ closest to $y$ among all even multiples $k\pi$'s with $|k|\le 2(\nu-1)$. If $y$ is not an odd multiple of $\pi$, then $k(y)$ is uniquely defined. So, $z(y):=y-k(y)\pi$ is uniquely defined for almost all $y$, and for those $y$  we have  $\phi(y)=\phi(z(y))$, since $\phi$ is $2\pi$-periodic. Furthermore, $|z(y)|$ is the distance from those $y$ to the set of even multiples of $\pi$, and as such it has a continuous extension to all $y$. Therefore, $|z(y)|< \pi$ for almost all $y$. Slightly abusing notation, from now on
let $[-\pi,\pi]^{\nu}$ stand for the original cube {\it minus\/} the set of $\bold x$ such that at least one $y_{\a}(\bold x)$ is an odd multiple of $\pi$. The discarded subset has zero Lebesgue measure. Thus, for $\bold x$ in the (reduced) cube  $[-\pi,\pi]^{\nu}$, we have: for each $\a\in [\nu]$, $z(y_{\a}(\bold x))=y_{\a}(\bold x)-k(y_{\a}(\bold x))\pi$ is well defined, and
denoting $z_{\a}=z(y_{\a}(x))$, $k_{\a}=k(y_{\a}(x))$, 
\begin{equation}\label{new3.3}
\Big|\Bbb E\bigl[f(\bold x,X)\bigr]\Big|\le\sum_{\a\in [\nu]}|\phi(z_{\a})|,\quad \sum_{\a\in [\nu]}z_{\a}=-\pi\sum_{\a\in [\nu]}k_{\a}.
\end{equation}

So, we need  to upper bound $|\phi(z_{\a})|$, $(\a\in [\nu])$,  for each of these subsets. Let $z\in [-\pi,\pi]$. Pick $b>1$. If $|\sin(z/2)|\ge 
b/M$, then $|\phi(z)|\le b^{-1}$. Also, $|\phi(z)|\le \pi/(M|z|)$,
because $\sin\eta\ge 2\eta/\pi$ for $\eta\in (0,\pi/2)$. 
Now $|\sin(z/2)|\le b/M$ if and only if $|z|\le b_0(M)/M$, where $b_0(M)$ is uniquely defined for $M\ge b$ by the condition $\sin(b_0/(2M))=b/M$ with $b_0/(2M)\in (0,\pi/2]$. A closer look shows that, given $b$,
\begin{equation}\label{3.35}
b_0(M)=2b \bigl[1+O(b^2/M^2)\bigr],\quad M\to\infty.
\end{equation}
Therefore $|\phi(z)|\le b^{-1}$ if $|z|\ge b_0(M)/M$, implying that
\begin{equation}\label{new3.4}
|\phi(z)|\le \min\bigl\{\pi/(M|z|); b^{-1}\bigr\},\quad\text{for }|z|\ge b_0(M)/M. 
\end{equation}
Consider $|z|\le b_0(M)/M$, and set $z=\xi/M$, i.e., $|\xi|\le b_0(M)\sim 2b$ for large $M$. Then
\[
|\phi(z)|=\Big|\frac{\sin(\xi/2)}{M\sin(\xi/(2M))}\Big|
=|\psi(\xi)|\cdot \bigl|1+O(\xi^2/M^2)\bigr|,\quad \psi(\xi):=\frac{\sin(\xi/2)}{\xi/2}.
\]
Since 
\[
\psi^2(\xi)=\frac{1-\cos(\xi)}{\xi^2/2}=\sum_{j\ge 1}\tfrac{2(-1)^{j-1}}{(2j)!}\xi^{2(j-1)},
\]
which is an alternating series with decreasing terms for $|\xi|\le 3\sqrt{2}$, we obtain that
$
\psi^2(\xi)\le 1-\tfrac{\xi^2}{30},\quad (|\xi| \le 3\sqrt{2}).
$
Therefore, $|\phi(z)|^2\le 1-\tfrac{\xi^2}{31}$ for $|\xi|\le 3\sqrt{2}$ and $M$ sufficiently large.
Furthermore, $\psi^2(\xi)\le \tfrac{2}{9}$ for $|\xi|\ge 3\sqrt{2}$. So, given $b$, there exists $M_1(b)$ such that $|\phi(z)|^2\le \tfrac{1}{3}$
for $|\xi|\in [3\sqrt{2}, b_0(M)]$ and $M\ge M_1(b)$. Invoking \eqref{3.35}, we have: there exists $M(b)\ge M_1(b)$ such that for $M\ge M(b)$ 
\begin{equation}\label{new3.5}
|\phi(z)|^2\le 1-\la(b)\xi^2,\quad \xi:=Mz,\,|\xi|\le b_0(M),\quad \la:=\min\{\tfrac{1}{31},\tfrac{1}{7b^2}\}.
\end{equation}
It follows from \eqref{new3.4} and \eqref{new3.5} that $|\phi(z)|<1$ for $0<|z|\le \pi$. Consequently, by \eqref{new3.2}, \eqref{new3.3}, we have $\Big|\Bbb E\bigl[f(\bold x,X)\bigr]\Big|<\nu$, unless $z_{\a}(\bold x)\equiv 0$, $(\a\in [\nu])$. Moreover, using Cauchy-Schwartz inequality, and $1-\ga\le e^{-\ga}$, we obtain
\begin{lemma}\label{lem2.5} Let $\bold x\in [-\pi,\pi]^{\nu}$. With $z_{\a}=z_{\a}(\bold x)$, denote $A(\bold x)=\{\a\in [\nu]:\, |z_{\a}|\le b_0(M)/M\}$,  $a(\bold x):=|A(\bold x)|$, and define $\xi_{\a}=\xi_{\a}(\bold x):=Mz_{\a}$, $\a\in [\nu]$. Given $b>0$, for $M\ge M(b)$ there exists $\la=\la(b)>0$ such that
\begin{multline}\label{star2}
 \Big|\Bbb E\bigl[f(\bold x,X)\bigr]\Big|^2\le \nu\biggl(\sum_{\a\in A(\bold x)}|\phi(z_{\a})|^2+\sum_{\a\in A^c(\bold x)}\!\!\!{\min}^2\bigl\{\tfrac{\pi}{|\xi_{\a}|}; b^{-1}\bigr\}\biggr)\\
\le\nu\biggl[ a(\bold x)\exp\Bigl(-\tfrac{\la}{a(\bold x)}\sum_{\a\in A(\bold x)}\xi_{\a}^2\Bigr)+
 \sum_{\a\in A^c(\bold x)}\!\!\!{\min}^2\bigl\{\tfrac{\pi}{|\xi_{\a}|}; b^{-1}\bigr\}\biggr].
\end{multline}
if $A(\bold x)\neq\emptyset$; otherwise the first term within the brackets  is set zero.
\end{lemma}
{\bf Note.\/} Of course, a similar inequality holds for  $\Big|\Bbb E\bigl[f(\bold x,X)\bigr]\Big|$, but  with 
{\it unsquared} ${\min}\bigl\{\tfrac{\pi}{|\xi_{\a}|}; b^{-1}\bigr\}$ and without factor $\nu$ on the RHS. The inequality \eqref{star2} has a useful advantage: ${\min}^2\bigl\{\tfrac{\pi}{|\xi|}; b^{-1}\bigr\}$ is integrable on $(-\infty,\infty)$. 
There is a price to pay though: to get a
bound for $\big|\Bbb E\bigl[f(\bold x,X)\bigr]\big|^n$, and then to integrate it over a peripheral part of the (reduced) cube $[-\pi,\pi]^{\nu}$ in question, we will have to raise both sides of the above inequality to the $(n/2)$-th power. So, while working on upper bounds, we assume that $n$ is even, and define $N=n/2$. For odd $n$ we would raise the RHS to the integer power $N\!:=\!(n-1)/2$ and multiply the result by $\nu\ge\Big|\Bbb E\bigl[f(\bold x,X)\bigr]\Big|$. }

{ The next step is to use Lemma \ref{lem2.5} in order to upper bound contribution to $\Bbb E[Z_n]$ coming from a vast part of the cube $[-\pi,\pi]^{\nu}$.
\begin{lemma}\label{lem2.55}  
Define $D=D_1\cup D_2$,
\begin{align*}
D_1&=\Bigl\{\bold x\in [-\pi,\pi]^{\nu}:\,a(\bold x)<\nu\Bigr\},\\
D_2&=\biggl\{\bold x\in [-\pi,\pi]^{\nu}:\,a(\bold x)=\nu,\,\,\sum_{\a\in [\nu]}\xi^2_{\a}\ge \tfrac{\log ^2 n}{n} \biggr\}.
\end{align*}
where $\xi_{\a}=Mz_{\a}$, $z_{\a}=y_{\a}(\bold x)-k_{\a}\pi$, $\min_{\text{even }k}|y_{\a}(\bold x)-k\pi|=
|y_{\a}(\bold x)-k_{\a}\pi|$. (Notice at once that 
\begin{multline*}
D^c=\biggl\{\bold x\in [-\pi,\pi]^{\nu}: a(\bold x)=\nu,\sum_{\a\in [\nu]}\xi_{\a}^2<\tfrac{\log^2 n}{n}\biggr\}\\\
=\biggl\{\bold x\in [-\pi,\pi]^{\nu}:\sum_{\a\in [\nu]}\xi_{\a}^2< \tfrac{\log^2 n}{n}\biggr\},
\end{multline*}
the last equality holding if $\tfrac{\log^2n}{n}\le b_0(M)(=O(b))$.) For $M=e^{O(n)}\to\infty$, and $N=n/2$, within factor $c(b)$ we have
\begin{multline}\label{3.501}
I_n:=\int\limits_{D}\big|\Bbb E\bigl[f(\bold x,X)\bigr]\big|^{n}\,d\bold x\\
\quad \le\tfrac{\nu^{2N}}{M^{\nu-1}}\exp\Bigl(-\tfrac{\la\log^2 N}{4\nu}\Bigr)
+\left\{\begin{aligned}
&\tfrac{(\nu e^{O(b^{-2})})^N}{M}, &&\lim\tfrac{N}{\log M}<1,\\
&\Bigl(\tfrac{\nu Ne^{O(b^{-2})}}{e\log M}\Bigr)^N,&& \lim\tfrac{N}{\log M}\in (1,\nu-1),\\
&\tfrac{\bigl(\nu(\nu-1)e^{O(b^{-2})}\bigr)^N}{M^{\nu-1}},&&\lim\tfrac{N}{\log M}>\nu-1.\end{aligned}\right.
\end{multline}
\end{lemma}
\begin{proof} 


{\bf (i)\/} Consider $\bold x\in D_1$, so that $|A(\bold x)|\in [0,\nu-1]$. Let $I_{n,1}$ ($I_{n,1}(a)$ resp.) denote the contribution of $D_1$ (of $D_1(a):=\{\bold x\in D_1: |A(\bold x)|=a\}$ resp.) to $\Bbb E[Z_n]$.  For $a>0$ and $\bold x\in D_1(a)$, we use \eqref{star2}
and a union-type bound to get
\begin{multline*}%
\big|\Bbb E\bigl[f(\bold x,X)\bigr]\big|^{n}\le\nu^N \sum_{A\subset [\nu]: |A|=a}\biggl[a\exp\Bigl(\!-\tfrac{\la}{a}\sum_{\a\in A}\xi_{\a}^2\Bigr)
\cdot I(\max_{\a\in A}|\xi_{\a}|\le b_0)\\
+\sum_{\be\in A^c}
\Bbb I(|\xi_{\be}|\in [b_0,\pi M])\,{\min}^2\bigl\{\tfrac{\pi}{|\xi_{\be}|}; b^{-1}\bigr\}\!\biggr]^N,\quad \xi_{\ga}=\xi_{\ga}(\bold x),\,\,\ga\in [\nu].
\end{multline*}
Each of $\binom{\nu}{a}$ terms contributes equally to the integral of the RHS over $[-\pi,\pi]^{\nu}$. Therefore, by multinomial formula
with $\bold r:=(r_1, r_{a+1},\dots, r_{\nu})$,
\begin{multline}\label{3.502}
 I_{n,1}(a)\le\nu^N\binom{\nu}{a}\sum_{\|\bold r\|=N}\frac{N!}{r_1!\,r_{a+1}!\,\cdots r_{\nu}!}\\
\times\!\!\!\int\limits_{\bold x\in [-\pi,\pi]^{\nu}}\biggl[a^{r_1}\exp\Bigl(-\tfrac{\la r_1}{a}\sum_{\a\in [a]}\xi_{\a}^2\Bigr)\cdot I(\max_{\a\in [a]}|\xi_{\a}|\le b_0)\\
\times\prod_{\be=a+1}^{\nu}
\Bbb I(|\xi_{\be}|\in[b_0,\pi M])\,{\min}^{2r_{\be}}\bigl\{\tfrac{\pi}{|\xi_{\be}|}; b^{-1}\bigr\}\biggr]\,d\bold x.
\end{multline}
Recall that 
$\{\xi_{\a}\}=\{\xi_{\a}(\bold x)\}$ is the piecewise affine function, with the same matrix of rank $\nu-1$ on each of at most 
$(2\nu-1)^{\nu}$ disjoint subsets forming a partition of $[-\pi,\pi]^{\nu}$.  
To bound the integral in \eqref{3.502} we will switch from $x_{\a}$ to $\xi_{\a}$ for all but a single component $x_{\mu}$, choosing $\mu$ {\it dependent on\/} $\bold r$.
No matter what $\mu$ is, the Jacobian factor is the same,
namely $(M^{\nu-1}\nu^{\nu-2})^{-1}$, on every subset of the resulting partition.
Let $\be\in \{a+1,\dots,\nu\}$;  if $\xi_{\be}$ is among the new variables, then--upon integration of $\bold r$-th term--variable $\xi_{\be}$ contributes a factor
\begin{equation}\label{3.503}
2\int\limits_{b_0}^{\pi M}\!\!\Bigl(\!\min\bigl\{\tfrac{\pi}{\xi}; b^{-1}\!\bigr\}\!\Bigr)^{2r_{\be}}\!\!d\xi\!\left\{\begin{aligned}
\!\!&=\!2(\pi M-b_0)\!\le\! 2 \pi M,&&r_{\be}=0,\\
\!\!&\le\!\tfrac{2(\pi b-b_0)}{b^{2r_{\be}}}+\tfrac{2\pi}{b^{2r_{\be}-1}(2r_{\be}-1)}\!\le\! \tfrac{4\pi}{b^{2r_{\be}-1}},&&r_{\be}\ge 1.\end{aligned}\right.
\end{equation}
We will control the factors for $r_{\be}\ge 1$ by choosing $b$ large. The case $r_{\be}=0$ is the crux of the matter. Since $a\ge 1$, the number of zeros in $\{r_{a+1},\dots, r_{\nu}\}$ can be as high as $\nu-a$, if $\mu\in A$, thus leading to factor $M^{\nu-a}$ in a bound
for the product of these integrals. However, we may and will select $\mu\in A^c$, if $r_{\mu}$ is one of the zeros in $\{r_{a+1},\dots, r_{\nu}\}$.  For this choice, 
$2\int_{b_0}^{\pi M} \Bigl(\min\bigl\{\tfrac{\pi}{\xi_{\mu}};b^{-1}\bigr\}\Bigr)^{2r_{\mu}} d\xi_{\mu}\approx 2\pi M$ is replaced with 
$\int_{-\pi}^{\pi} 1\, dx_{\mu}=2\pi$. Hence, the worst-case bound becomes $M^{\nu-a-1}$ instead!  Only if there are no zeros in 
$\{r_{a+1},\dots, r_{\nu}\}$ 
do we choose $\mu\in A$.

Since $\mu\le \nu$, 
we obtain
\begin{equation}\label{3.504}
\begin{aligned}
 I_{n,1}(a)&\le \nu^N\binom{\nu}{a}\tfrac{(2\nu-1)^{\nu}\nu}{M^{\nu-1}\nu^{\nu-2}}\sum_{\|\bold r\|=N}\frac{N!\, a^{r_1}}{r_1!\,r_{a+1}!\cdots r_{\nu}!} \cdot J(\bold r),\\
J(\bold r)&:=\int\limits_{x_{\mu},\,\{\xi_{\ga}\}_{\ga \neq \mu}}\!\!\!\!\exp\Bigl(-\tfrac{\la r_1}{a}\sum_{\a\in [a]}\xi_{\a}^2\Bigr)
\cdot I(\max_{\a\in [a]}|\xi_{\a}|\le b_0)\\
&\times\prod_{\be=a+1}^{\nu}
\Bbb I(|\xi_{\be}|\in[b_0,\pi M])\,{\min}^{2r_{\be}}\bigl\{\tfrac{\pi}{|\xi_{\be}|}; b^{-1}\bigr\}\,\,dx_{\mu}\prod_{\ga\neq\mu} d\xi_{\ga}.
\end{aligned}
\end{equation}
Let $\sigma(\bold r)$ stand for the number of zeros in $\{r_{a+1},\dots, r_{\nu}\}$. If $\sigma(\bold r)>0$, then we choose $\mu\in \{a+1,\dots, \nu\}$ such that $r_{\mu}=0$. Rather crudely, we have
\begin{equation*}
\int\limits_{\xi_1,\dots,\xi_a}\!\!\!\!\exp\Bigl(-\tfrac{\la r_1}{a}\sum_{\a\in [a]}\xi_{\a}^2\Bigr)\cdot I(\max_{\a\in [a]}|\xi_{\a}|\le b_0) \prod_{\a\in [a]}d\xi_{\a}\le (2b_0)^{a}=O(b^a).
\end{equation*}
Furthermore, using \eqref{3.503} and
\[
\sum_{\be\in [a+1,\,\nu]:\, r_{\be}>0}(2r_{\be}-1)=2(N-r_1)-\bigl[\nu-a-\sigma(\bold r)\bigr],
\]
we bound
\begin{multline*}
\int\limits_{x_{\mu},\,\{\xi_{\ga}\}_{\mu\neq \ga>a}}\prod_{\be=a+1}^{\nu}
\Bbb I(|\xi_{\be}|\in[b_0,\pi M])\,{\min}^{2r_{\be}}\bigl\{\tfrac{\pi}{|\xi_{\be}|}; b^{-1}\bigr\}\,\,dx_{\mu}\prod_{\mu\neq \ga> a} d\xi_{\ga}\\
\le 2\pi\cdot 2^{\nu-a-1}\prod_{\mu\neq \ga>a}\int\limits_{b_0}^{\pi M}\!\!\Bigl(\!\min\bigl\{\tfrac{\pi}{\xi}; b^{-1}\!\bigr\}\!\Bigr)
^{2r_{\ga}}\!\!d\xi\!\\
\le 2^{\nu-a}\pi (\pi M)^{\sigma(\bold r)-1} (4\pi)^{\nu-a-\sigma(\bold r)} b^{-2(N-r_1)+(\nu-a-\sigma(\bold r))}\\
\le c_1M^{\sigma(\bold r)-1} b^{-2(N-r_1)}\le c_1M^{\nu-a-1} b^{-2(N-r_1)}.\\
\end{multline*}
Here $c_1$, and $c_j$ below depend on $b$ only. So, $J(\bold r)\le c_2M^{\nu-a-1} b^{-2(N-r_1)}$ if $\sigma(\bold r)>0$. 
Likewise, if $\sigma(\bold r)=0$ then $J(\bold r)\!\le \!c_3 b^{-2(N-r_1)}\!\le\! c_3M^{\nu-a-1} b^{-2(N-r_1)}$, since
$\nu>a$. 
Therefore $J(\bold r)\le c_4M^{\nu-a-1} b^{-2(N-r_1)}$. Plugging this bound into \eqref{3.504} we have: for $a\in [1,\nu-1]$,
\begin{multline}\label{3.8new}
I_{n,1}(a)\le 
\tfrac{c_5\nu^N}{M^{\nu-1}}
 \sum_{\| \bold r\|=N}
\frac{M^{\nu-a-1} N! }{\prod\limits_s r_s!}
\cdot a^{r_1} b^{-2(N-r_1)}\\
=\tfrac{c_5\nu^N}{M^{\nu-1}}\cdot M^{\nu-a-1}\sum_{\|\bold r\|=N}\frac{N!}{ \prod\limits_s r_s!}\,a^{r_1}\prod_{s>a} b^{-2r_s}\\
\le\tfrac{c_5\nu^N}{M^{\nu-1}}\cdot M^{\nu-a-1}\bigl(a+\nu b^{-2}\bigr)^N.
\end{multline}

For $a=0$, instead of the bound \eqref{3.502} we get
\begin{multline*}
\big|\Bbb E\bigl[f(\bold x,X)\bigr]\big|^{n}\le\nu^N \biggl(\sum_{\be=1}^{\nu}
\Bbb I(|\xi_{\be}|\in [b_0,\pi M])\,{\min}^2\bigl\{\tfrac{\pi}{|\xi_{\be}|}; b^{-1}\bigr\}\!\biggr)^N\\
=\nu^N\sum_{\|\bold r\|=N}\frac{N!}{r_1!\cdots r_{\nu}!}
\prod_{\be=1}^{\nu}
\Bbb I(|\xi_{\be}|\in[b_0,\pi M])\,{\min}^{2r_{\be}}\bigl\{\tfrac{\pi}{|\xi_{\be}|}; b^{-1}\bigr\},
\end{multline*}
$\bold r=(r_1,\dots,r_{\nu})$. And, analogously to \eqref{3.8new}, we obtain 
\begin{equation}\label{3.81new}
 I_{n,1}(0)\le  \tfrac{c_6\nu^N}{M^{\nu-1}}\cdot M^{\nu-2} b^{-2N}.
\end{equation}
Why $M^{\nu-2}$? Because $\sigma(\bold r)\le\nu-1$, and if $\sigma(\bold r)=\nu-1$, then  we select $\mu$ for which $r_{\mu}=0$.

Adding \eqref{3.8new} and \eqref{3.81new}, we get
\begin{equation}\label{star'}
I_{n,1}\le 
\tfrac{c_7\nu^N}{M}\biggl(b^{-2N}+\sum_{a=1}^{\nu-1}M^{-a+1}(a+\nu b^{-2})^N\biggr).
\end{equation}
Let us simplify \eqref{star'}. For $n\to\infty$, the (log-concave) function $\psi(a):=M^{-a+1}\bigl(a+\tfrac{\nu}{b^2})^N$, $(a\in [1,\nu-1])$, attains its unique maximum at 
\begin{equation*}
a(\nu)=\left\{\begin{aligned}
&1,&&\lim\tfrac{N}{\log M}<1+\tfrac{\nu}{b^2},\\
&\tfrac{N}{\log M}-\tfrac{\nu}{b^2},&&\lim\tfrac{N}{\log M}\in (1+\tfrac{\nu}{b^2},\nu-1+\tfrac{\nu}{b^2}),\\
&\nu-1,&&\lim\tfrac{N}{\log M}>\nu-1+\tfrac{\nu}{b^2}.\end{aligned}\right.
\end{equation*}
Consequently
\begin{equation*}
\psi(a(\nu))=\left\{\begin{aligned}
&\bigl(e^{O(b^{-2})}\bigr)^N,&&\lim\tfrac{N}{\log M}<1+\tfrac{\nu}{b^2},\\
&M\Bigl(\tfrac{Ne^{O(b^{-2})}}{e\log M}\Bigr)^N,&&\lim\tfrac{N}{\log M}\in (1+\tfrac{\nu}{b^2},\nu-1+\tfrac{\nu}{b^2}),\\
&M^{-\nu+2}\bigl((\nu-1)e^{O(b^{-2})}\bigr)^N,&&\lim\tfrac{N}{\log M}>\nu-1+\tfrac{\nu}{b^2}.\end{aligned}\right.
\end{equation*}
The sum in \eqref{star'} is below $b^{-2N}+\nu \psi(a(\nu))$, and $b$ can be chosen arbitrarily 
large. So, for $M=e^{O(n)}\to\infty$ and $b$ sufficiently large, we have
\begin{equation}\label{new3.719}
I_{n,1}\le c_8\cdot\left\{\begin{aligned}
&\tfrac{\bigl(\nu e^{O(b^{-2})}\bigr)^N}{M}, &&\lim\tfrac{N}{\log M}<1,\\
&\Bigl(\tfrac{\nu Ne^{O(b^{-2})}}{e\log M}\Bigr)^N,&& \lim\tfrac{N}{\log M}\in (1,\nu-1),\\
&\tfrac{\bigl(\nu(\nu-1)e^{O(b^{-2})}\bigr)^N}{M^{\nu-1}},&&\lim\tfrac{N}{\log M}>\nu-1.\end{aligned}\right.
\end{equation}

{\bf (ii)\/} Suppose  $\bold x\in D_2$. Let $I_{n,2}$ denote the contribution of $D_2$ to $\Bbb E[Z_n]$.
We switch from $x_{\a}$ to $\xi_{\a}$ for all but a single $x_{\mu}$ such that $|\xi_{\mu}|=\min_{\a}|\xi_{\a}|$.
So, here the cube $[-\pi,\pi]^{\nu}$ is a finite disjoint union of subsets, and the choice of the new variables depends  
on a subset in such a way  that the condition $\sum_{a\in [\nu]}\xi_{\a}^2\ge \tfrac{\log^2n}{n}$ implies that
 $\sum_{a\neq \mu}\xi_{\a}^2\ge \tfrac{(\nu-1)\log^2n}{\nu n}\ge \tfrac{\log^2N}{4N}$.

Analogously to $I_{n,1}$, by Lemma \ref{lem2.5} we have 
\begin{equation*}
I_{n,2}\le \tfrac{c_1\nu^{2N}}{M^{\nu-1}}\!\!\!\!\!\!\!\!\!\!\!\!\int\limits_{\sum_{\a=1}^{\nu-1}\xi_{\a}^2\ge \tfrac{\log^2N}{4N}}
\!\!\!\!\!\!\!\!\!\!\exp\Bigl(-\tfrac{\la N}{\nu}\sum_{\a=1}^{\nu-1}\xi_{\a}^2\Bigr)\,\prod_{\a=1}^{\nu-1} d\xi_{\a}.
\end{equation*}
Here, and below, $c_j$ is independent of $b$ and $\la=\la(b)$. Using spherical coordinates, and denoting $\zeta_0=
\bigl(\tfrac{\la}{4\nu}\bigr)^{1/2}\log N$, we bound the last integral by
\begin{align*}
c_2\!\!\!\int\limits_{\rho\ge \frac{\log N}{\sqrt{4N}}}\!\!\!\exp\Bigl(-\tfrac{\la N}{\nu}\rho^2\Bigr) \rho^{\nu-2}\,d\rho
&=c_3\,(\la N)^{-(\nu-1)/2}\!\!\int\limits_{\zeta\ge \zeta_0}e^{-\zeta^2/2}\zeta^{\nu-2}\,d\zeta\\
&\le c_4\, \la^{-1} N^{-(\nu-1)/2} (\log N)^{\nu-3} \exp\Bigl(-\tfrac{\la\log^2 N}{4\nu}\Bigr).
\end{align*}
The last bound follows from integrating the LHS integral (call it $I$) by parts once, and bounding the residual integral  by $I\cdot(\nu-3)\zeta_0^{-2}$. We conclude that
\begin{equation}\label{new3.72}
 I_{n,2}\le \tfrac{c_5\,\nu^{2N}}{\la M^{\nu-1}}\exp\Bigl(-\tfrac{\la\log^2 N}{4\nu}\Bigr).
\end{equation}
Adding 
\eqref{new3.719} and \eqref{new3.72} we complete the proof of Lemma \ref{lem2.55}.
\end{proof}
}

{\bf (III)\/} {So, it remains to sharply evaluate an asymptotic contribution to the integral 
\[
\Bbb E\bigl[Z_n\bigr]
=\!\frac{1}{(2\pi)^{\nu}}\!\!\!\int\limits_{\bold x\in [-\pi,\pi]^{\nu}}\!\!\!\Bbb E^n\bigl[
f(\bold x,X) \bigr] d\bold x,
\]
that comes from $C:=[-\pi,\pi]^{\nu}\setminus D$, or more explicitly from $\bold x$'s in the (reduced) cube $[-\pi,\pi]^{\nu}$ with $\sum_{\a}\xi_{\a}^2<n^{-1}\log^2 n$,
where $\xi_{\a}=\xi_{\a}(\bold x)$ are uniquely defined by
\[
\xi_{\a}=Mz_{\a},\quad z_{\a}=y_{\a}(\bold x)-k_{\a}(\bold x)\pi,\quad |z_{\a}|=\min_{\text{even }k}|y_{\a}(\bold x)-k\pi|.
\] 
That is, $\{k_{\a}(\bold x)\}$ is an even tuple. To begin,
\begin{align*}
&\sum_{\a} \bigl(y_{\a}(\bold x)-k_{\a}(\bold x)\pi\bigr)^2=M^{-2}\sum_{\a}\xi_{\a}^2
<M^{-2}n^{-1}\log^2 n\\
&\qquad\qquad\Longrightarrow \sum_{\a}k_{\a}(\bold x)=0\Longrightarrow\sum_{\a}\xi_{\a}=0,
\end{align*}
{\it if $n$ is large enough\/}. Indeed $\sum_{\a}y_{\a}(\bold x)\equiv 0$; so if $\sum_{\a}k_{\a}(\bold x)\neq 0$, then
$\Big|\sum_{\a}k_{\a}(\bold x)\Big|\ge 2$, and for large $n$ we have a contradiction:
\[
4\pi^2\le \Bigl|\sum_{\a}(y_{\a}(\bold x)-k_{\a}(\bold x)\pi)\Bigr|^2\le\nu\sum_{\a}(y_{\a}(\bold x)-k_{\a}(\bold x)\pi)^2\le\tfrac{\nu\log^2n}{M^2 n}.
\]
Conversely, if $\{y_{\a}(\bold x)\}$ and an even tuple
$\bold k=\{k_{\a}\}$, $\bigl(\sum_{\a}k_{\a}=0\bigr)$, satisfy\linebreak $\sum_{\a} \bigl(y_{\a}(\bold x)-k_{\a}\pi\bigr)^2<M^{-2}n^{-1}\log^2 n$,
then $\sum_{\a}\xi_{\a}^2\le n^{-1}\log^2 n$. Moreover, $\xi_{\a}=M(y_{\a}(\bold x)-k_{\a}\pi)$, if $n$ is large.  Indeed,  $\xi_{\a}=M(y_{\a}(\bold x)-k_{\a}'\pi)$, where $|y_{\a}(\bold x)-k_{\a}'\pi)|=\min_{\text{even }\kappa}|y_{\a}(\bold x)-\kappa\pi|$, and the conditions
\[
\sum_{\a} \bigl(y_{\a}(\bold x)-k_{\a}\pi\bigr)^2,\quad\sum_{\a} \bigl(y_{\a}(\bold x)-k_{\a}'\pi\bigr)^2<M^{-2}n^{-1}\log^2 n,
\]
combined with the triangle inequality, imply that $\|\bold k-\bold k'\|\le \tfrac{2\log n}{ \pi M n^{1/2}}$.

 Therefore, $C=\cup_{\bold k}C_{\bold k}$ where $\bold k=\{k_{\a}\}$, with even $k_{\a}$ adding up to $0$, and
\[
C_{\bold k}:=\Bigl\{\bold x\in [-\pi,\pi]^{\nu}:\sum_{\a} \bigl(y_{\a}(\bold x)-k_{\a}\pi\bigr)^2<M^{-2}n^{-1}\log^2 n\Bigr\}.
\]
And, for large $n$, by the triangle inequality $C_{\bold k_1}\cap\, C_{\bold k_2}=\emptyset$ if $\bold k_1\neq \bold k_2$.

Consider the line $\mathcal L_{\bold k}$ given by $y_{\a}(\bold x)=k_{\a}\pi$, $(\a\in [\nu])$, for any such tuple $\bold k$. 
(In the part {\bf (I)\/} we enumerated all $\mathcal L_{\bold k}$ that contain interior points of $[-\pi,\pi]^{\nu}$.) A generic $\mathcal L_{\bold k}$ is given by its parametric equation 
\begin{equation}\label{new3.71}
x_{\a}(t)=\nu^{-1}(t+k_{\a}\pi),\quad \a\in [\nu].
\end{equation}
The lines are parallel to each other, running in the direction of $\bold e=(1,\dots,1)$, and crossing at $90$ degrees each of the (hyper)planes $\sum_{\a\in [\nu]}x_{\a}=t$, with $\bold x(t):=\{x_{\a}(t)\}_{\a\in [\nu]}$ being the
intersection point.  
Using 
$
x_{\a}=\nu^{-1}\Bigl(\sum_{\be}x_{\be}+y_{\a}(\bold x)\Bigr),
$
and 
\eqref{new3.71}, we obtain: if $\bold x=\{x_{\a}\}\in C_{\bold k}$ 
 then, with $t:=\sum_{\a}x_{\a}$,  
\begin{equation}\label{new3.715}
\begin{aligned}
x_{\a}-x_{\a}(t)&=\tfrac{y_{\a}(\bold x)-k_{\a}\pi}{\nu}=\tfrac{\xi_{\a}(\bold x)}{M\nu}\\
&\Longrightarrow \|\bold x-\bold x(t)\|^2=\tfrac{1}{(M\nu)^2}\sum_{\a}\xi_{\a}^2(\bold x)\le \tfrac{\log^2 n}{(M\nu)^2n}.
\end{aligned}
\end{equation}
Conversely, if $t=\sum_{\a}x_{\a}$, then $\|\bold x-\bold x(t)\|^2\le \tfrac{\log^2 n}{(M\nu)^2n}$ implies that $\bold x\in C_{\bold k}$.

Consider a $\nu$-dimensional cylinder $\mathcal C_{\bold k}$ enclosing the line $\mathcal L_{\bold k}$, such that each of its cross-sections formed by planes $\sum_{a}x_{\a}=t$ {\it orthogonal\/} to  $\mathcal L_{\bold k}$ is a $(\nu-1)$-dimensional (hyper)sphere in that plane, of radius $r_n:=
(M\nu)^{-1}n^{-1/2}\log n$, which is centered at the common point $\bold x(t)$ of the plane and the line $\mathcal L_{\bold k}$. We denote 
the sphere $\mathcal S(\bold x(t), r_n)$. 

The above discussion 
means the following.
If $\mathcal L_{\bold k}$ contains interior points of $[-\pi,\pi]^{\nu}$,
then $C_{\bold k}=\mathcal C_{\bold k}\cap [-\pi,\pi]^{\nu}$. And shortly we will see that the contribution of $C_{\bold k}$ to $\Bbb E[Z_n]$ is 
of exact order
$
L(\bold k)\nu^n M^{-\nu+1} n^{-\tfrac{\nu-1}{2}},
$
where $L(\bold k)\ge 2\pi \nu^{-1/2}$ is the length of line segment of $\mathcal L_{\bold k}$ in $[-\pi,\pi]^{\nu}$, see \eqref{3.016}. 

In the part {\bf (I)\/} we proved that if $\mathcal L_{\bold k}$ contains no interior points of $[-\pi,\pi]^{\nu}$, then $\mathcal L_{\bold k}$ may only touch $[-\pi,\pi]^{\nu}$ at a single point $\bold x_{\bold k}=\{x_{\bold k,\a}\}$, such that $x_{\bold k,\a_1}=\pi$ and $x_{\bold k,\a_2}=-\pi$ for some $\a_1,\a_2\in [\nu]$.  In this case, $C_{\bold k}$ is the union 
of the parallel line segments $\bold x(t)=\tilde {\bold x}_{\bold k} +t\bold e \in [-\pi,\pi]^{\nu}$, where $\sum_{\a}\tilde{\bold x}_{\bold k,\a}=\sum_{\a}x_{\bold k,\a}$, and $\|\tilde{\bold x}_{\bold k}-\bold x_{\bold k}\|\le r_n$. Using $x_{\bold k,\a_1}=\pi$ and $x_{\bold k,\a_2}=-\pi$ we obtain that $t\in [x_{\bold k,\a_2}-\tilde x_{\bold k,\a_2},
x_{\bold k,\a_1}-\tilde x_{\bold k, \a_1}]$. Since $|y|$ is convex, we see then that $|t|\le r_n$.
So, $C_{\bold k}$ is enclosed in a sub-cylinder of $\mathcal C_{\bold k}$ sandwiched between two parallel cross-sections, for $t_1=\sum_{\a} x_{\bold k,\a}-r_n$ and $t_2=\sum_{\a} x_{\bold k,\a}+r_n$, that are at distance $O(r_n)$ from each other.
Consequently, contribution of $C_{\bold k}$ to $\Bbb E[Z_n]$ is at most of order $r_n\nu^n M^{-\nu+1} n^{-\tfrac{\nu-1}{2}}$. And we know that the total number of the points $\bold x_{\bold k}$ is $\nu(2\nu-1)^{\nu}$, at most. So, 
the contribution to $\Bbb E[Z_n]$ from the sets $C_{\bold k}$ with the lines
$\mathcal L_{\bold k}$ merely touching the cube $[-\pi,\pi]^{\nu}$ is relatively negligible as $n\to\infty$. 
Therefore, we focus on $C_{\bold k}$ with lines $\mathcal L_{\bold k}$ enumerated in  {\bf (I)\/}. 
 
For $\bold x\in \mathcal S(\bold x(t), r_n)$, we set $\xi_{\a}=M(y_{\a}(\bold x)-k_{\a}\pi)=M\nu(x_{\a}-x_{\a}(t))$, so that $\sum_{\a}\xi_{\a}^2\le n^{-1}\log ^2n$ and $\sum_{\a}\xi_{\a}=0$. Then
\begin{multline*}
\Bbb E\bigl[
f(\bold x,X) \bigr] =\sum_{\a\in [\nu]} \Bbb E\bigl[e^{i y_{\a}(\bold x)X}\bigr]=\sum_{\a\in [\nu]} \Bbb E\bigl[e^{i \xi_{\a}M^{-1}X}\bigr]\\
=\sum_{\a\in [\nu]}\Bigl[1+i\tfrac{\Bbb E[X]}{M}\xi_{\a}-\tfrac{\Bbb E[X^2]}{2M^2}\xi_{\a}^2+O\bigl(|\xi_{\a}|^3\tfrac{\Bbb E[X^3]}{M^3}\bigr)\Bigr]\\
=\nu-\tfrac{\Bbb E[X^2]}{2M^2}\sum_{\a\in [\nu]}\xi_{\a}^2+O\Bigl(\tfrac{\Bbb E[X^3]}{M^3}\sum_{\a\in [\nu]}|\xi_{\a}|^3\Bigr)\\
=\nu\exp\Bigl(-\tfrac{c_{\scriptscriptstyle M}}{2\nu}\sum_{\a\in [\nu]}\xi_{\a}^2+O(n^{-3/2}\log^3n)\Bigr)\\
=\nu\exp\Bigl(-\tfrac{c_{\scriptscriptstyle M}}{2\nu}(M\nu)^2\|\bold x-\bold x(t)\|^2+O(n^{-3/2}\log^3n)\Bigr);
\end{multline*}
here $c_{\scriptscriptstyle M}=\tfrac{\Bbb E[X^2]}{M^2}$, $t=\sum_{\a}x_{\a}$.  So, uniformly for $\bold x\in \mathcal S(\bold x(t),r_n)$, we have
\begin{equation}\label{new3.716}
\Bbb E^n\bigl[f(\bold x,X) \bigr] =\nu^n\exp\Bigl(-\tfrac{c_{\scriptscriptstyle M}n}{2\nu}(M\nu)^2\|\bold x-\bold x(t)\|^2+O(n^{-1/2}\log^3n)\Bigr).
\end{equation}
With $\rho_n:=n^{-1/2}\log n\to 0$, admissibility issue arises when $|t-t_j|=O(M^{-1}\rho_n)$, where $t_1:=-\pi(\nu+\min k_{\a})$ and $t_2:=\pi(\nu-\max k_{\a})$ are the endpoints of $t$'s range.
Ignoring this lack of complete homogeneity, and using orthogonality of $\mathcal S(\bold x(t), r_n)$ and $\mathcal L_{\bold k}$, we replace the contribution of $C_{\bold k}$ to the integral representing $E[Z_n]$ with the integral of the RHS in \eqref{new3.716} over the sphere $\mathcal S(\bold x(t), r_n)$, times $C_{\bold k}$'s length
\[
L(\bold k):=\|\bold x(t_2)-\bold x(t_1)\|=\nu^{-1/2} \pi\bigl(2\nu+\min_{\a}k_{\a}-\max_{\a}k_\a\bigr)\ge 2\pi\nu^{-1/2},
\]
see \eqref{3.016}.} So, introducing $\mathcal S(\bold 0,\rho)$, the $(\nu-1)$-dimensional sphere of radius $\rho$, centered at the origin $\bold 0\in E^{\nu-1}$, and using
\[
\text{Volume}\bigl(\mathcal S(\bold 0,\rho)\bigr)=\tfrac{\pi^{\frac{\nu-1}{2}}\rho^{\nu-1}}{\Gamma(\tfrac{\nu-1}{2}+1)},
\]
we obtain: within factor $1+o(1)$, contribution of $C_{\bold k}$ to 
$\Bbb E[Z_n]$ is
\begin{multline*}
\frac{\nu^n L(\bold k)}{(2\pi)^{\nu}(M\nu)^{\nu-1}}\int_{\mathcal S(\bold 0, \rho_n)}e^{-\frac{c_M n}{2\nu}
\rho^2}\,dV\\
=\frac{\nu^n L(\bold k)}{(2\pi)^{\nu}(M\nu)^{\nu-1}}\cdot\frac{\pi^{\tfrac{\nu-1}{2}}\tfrac{\nu-1}{2}}{\Gamma\bigl(\tfrac{\nu-1}{2}+1\bigr)}\bigl(\tfrac{2\nu}{c_Mn}\bigr)^{\tfrac{\nu-1}{2}}
\int_0^{\tfrac{c_Mn\rho_n^2}{\nu}}e^{-z} z^{\tfrac{\nu-1}{2}-1}\,dz\\
=(1+o(1))\frac{\nu^n}{M^{\nu-1}}\cdot\frac{L(\bold k)}{2\pi(2\pi\nu c_{\scriptscriptstyle M}n)^{\tfrac{\nu-1}{2}}}.
\end{multline*}
Indeed, the second line integral is asymptotic to $\Gamma(\tfrac{\nu-1}{2})$, and for the third line we used $\Gamma(z+1)=z\Gamma(z)$. (This estimate confirms our earlier claim concerning the contribution of 
every such $C_{\bold k}$ to $\Bbb E[Z_n]$.)
Combining the estimate with Lemmas \ref{lem1.6} and \ref{lem2.55}, we have: 
\begin{equation}\label{clear1}
\Bbb E\bigl[Z_n\bigr]=(1+o(1))\frac{\nu^n}{M^{\nu-1}}\cdot\frac{\nu^{\nu-3/2}}{(2\pi\nu c_{\scriptscriptstyle M} n)^{\frac{\nu-1}{2}}}+O(\mathcal R_n),\\
\end{equation}
where (with $N=n/2$)
\begin{multline}\label{new3.717}
\mathcal R_n=\tfrac{\nu^{n}}{M^{\nu-1}}\exp\Bigl(-\tfrac{\la\log^2 N}{4\nu}\Bigr)\\
+\left\{\begin{aligned}
&\tfrac{\bigl(\nu e^{O(b^{-2})}\bigr)^N}{M}, &&\lim\tfrac{N}{\log M}<1,\\
&\Bigl(\tfrac{\nu Ne^{O(b^{-2})}}{e\log M}\Bigr)^N,&& \lim\tfrac{N}{\log M}\in (1,\nu-1),\\
&\tfrac{\bigl(\nu (\nu-1)e^{O(b^{-2})}\bigr)^N}{M^{\nu-1}},&&\lim\tfrac{N}{\log M}>\nu-1.\end{aligned}\right.\\
\end{multline}
We will refer to this vertical stack as top, middle, and bottom remainders.

{\bf (IV)\/} Let us have a close look at the estimate above. 

{\bf (a)\/} $\nu=3$. Suppose that $\lim\tfrac{N}{\log M}\in 
\bigl(\tfrac{1}{\log 3},\infty\bigr)\setminus\{1,2\}$. Then $\tfrac{3^{2N}}{M^2}$ grows exponentially fast, hence 
$\tfrac{3^{2N}}{M^2}\gg\tfrac{\bigl(3e^{O(b^{-2})}\bigr)^N}{M}$ (top remainder term in \eqref{new3.717}), for $b$ chosen sufficiently large. 
So, if $\lim\tfrac{N}{\log M}\in\bigl(\tfrac{1}{\log 3}, 1\bigr)$,  then the above formula for $\Bbb E\bigl[Z_n\bigr]$ implies that 
\begin{equation}\label{newclose}
\Bbb E\bigl[Z_n\bigr]=(1+o(1))\frac{\nu^n}{M^{\nu-1}}\cdot\frac{\nu^{\nu-3/2}}{(2\pi\nu c_{\scriptscriptstyle M} n)^{\frac{\nu-1}{2}}}
\Big|_{\nu=3}\to\infty.
\end{equation}
Let $\lim\tfrac{N}{\log M}\in (1,2)$;  \eqref{newclose} continues to hold for $b$ large enough, provided that
\[
\tfrac{3^{2N}}{M^2}\gg \Bigl(\tfrac{\nu Ne^{O(b^{-2})}}{e\log M}\Bigr)^N,
\]
middle remainder in \eqref{new3.717}. This inequality holds if  $\rho:=\lim \tfrac{\log M}{N}\in (0.5,1) $ satisfies $(3e^{-\rho})^2>
\tfrac{3}{e\rho}$. And it is indeed so, since
\[
\min\{9\tau e^{-2\tau}:\,\tau\in [1/2,1]\}=\tfrac{9}{e^2}>\tfrac{3}{e}.
\]
Let $\lim\tfrac{N}{\log M}\in (2,\infty)$; \eqref{newclose} again holds for $b$ large enough, since $3^n\gg 6^N$, see bottom remainder in \eqref{new3.717}. On the other hand, if $\lim\tfrac{N}{\log M}<\tfrac{1}{\log 3}$, then $\lim\tfrac{N}{\log M}<1$, and using top remainder
from \eqref{new3.717} we see that $\Bbb E\bigl[Z_n\bigr]=O\Bigl(\tfrac{(3 e^{O(b^{-2})})^N}{M}\Bigr)\to 0$, for $b$ sufficiently large. 

{\bf (b)\/} Let now $\nu>3$. \eqref{newclose} holds if $\lim\tfrac{N}{\log M}>\nu-1$. Indeed, by \eqref{new3.717}, 
\begin{align*}
&\quad\mathcal R_n=\tfrac{\nu^{2N}}{M^{\nu-1}}\exp\Bigl(-\tfrac{\la\log^2 N}{4\nu}\Bigr)+\tfrac{\bigl(\nu (\nu-1)e^{O(b^{-2})}\bigr)^N}{M^{\nu-1}},\\
&\tfrac{\nu^{2N}}{M^{\nu-1}}=\exp\Bigl[2\log M\cdot\log \nu\cdot\bigl(\tfrac{N}{\log M}-\tfrac{\nu-1}{2\log\nu}\bigr)\Bigr]\to\infty,\\
&\qquad\,\,\,\frac{\tfrac{\nu^{2N}}{M^{\nu-1}}}{\tfrac{\bigl(\nu(\nu-1)e^{O(b^{-2})}\bigr)^N}{M^{\nu-1}}}=\bigl(\tfrac{\nu}{\nu-1}e^{O(b^{-2})}\bigr)^N\to\infty,
\end{align*}
if $b$ is sufficiently large, since $\tfrac{\nu-1}{2\log\nu}<\nu-1$. 

Consider now  $\tfrac{\nu-1}{2\log\nu}< \lim\tfrac{N}{\log M}<\nu-1$. Since $\min_{\nu>3}\tfrac{\nu-1}{2\log\nu}=\tfrac{4-1}{2\log 4}>1$
(notice that $\tfrac{3-1}{2\log 3}<1$), we use middle remainder from \eqref{new3.717} to obtain
\[
\mathcal R_n=\tfrac{\nu^{2N}}{M^{\nu-1}}\exp\Bigl(-\tfrac{\la\log^2 N}{4\nu}\Bigr)+\Bigl(\tfrac{\nu Ne^{O(b^{-2})}}{e\log M}\Bigr)^N.
\]
Therefore
\begin{equation}\label{etan}
\frac{\mathcal R_n}{\tfrac{\nu^n}{M^{\nu-1}}}=\exp\Bigl(-\tfrac{\la\log^2 N}{4\nu}\Bigr)+\Bigl(\tfrac{e^{(\nu-1)\eta_n}+O(b^{-2})}{\nu e\eta_n}\Bigr)^N,
\quad \eta_n:=\tfrac{\log M}{N},
\end{equation}
i.e. $\lim\eta_n\in \bigl(\tfrac{1}{\nu-1},\tfrac{2\log\nu}{\nu-1}\bigr)$. The function $f_{\nu}(\eta):=\tfrac{e^{(\nu-1)\eta}}{\nu e\eta}$ increases with $\eta>(\nu-1)^{-1}$, and 
\[
f_{\nu}((\nu-1)^{-1})=1-\tfrac{1}{\nu}<1,\quad f_{\nu}\bigl(\tfrac{2\log\nu}{\nu-1}\bigr)=\tfrac{\nu(\nu-1)}{2e\log\nu}>1,\,\,\forall\nu>3.
\]
So, there exists a unique root $\eta(\nu)\in \bigl(\tfrac{1}{\nu-1},\tfrac{2\log\nu}{\nu-1}\bigr)$ of $f_{\nu}(\eta)-1$. Using \eqref{etan}, we conclude that for $\nu>3$ the equation 
\eqref{newclose} holds for $\tfrac{1}{\eta(\nu)}<\lim\tfrac{N}{\log M}<\nu-1$, if $b$ is chosen large enough. 

Finally, like the case $\nu=3$, $\lim \Bbb E[Z_n]=0$ if $\lim\tfrac{N}{\log M}<\tfrac{1}{\log\nu}$. This completes the proof of Theorem \ref{newthm1}. 

\end{proof}

 \section{Second order moment of the number of perfect partitions} From Theorem \ref{newthm1} we know that, for $\lim\tfrac{n}{\log M}<\tfrac{2}{\log \nu}$, we have $\Bbb E\bigl[Z_n\bigr]\to 0$, so that $\Bbb P(Z_n>0)
\le \Bbb E[Z_n]\to 0$. Furthermore, we stated the conditions on $\lim\tfrac{n}{\log M}$ under which $\lim\Bbb E[Z_n]=\infty$.
These conditions make it plausible, but certainly do not imply, that $\Bbb P(Z_n>0)$ approaches $1$, or even that $\lim \Bbb P(Z_n>0)>0$. 
In general, if the first order moment grows very fast, it may well portend that the standard deviation grows even faster. 

\begin{theorem}\label{newthm2} Suppose that $\lim\tfrac{n}{\log M}>\tfrac{2(\nu-1)}{\log[(1-2\nu^{-2})^{-1}]}$.
Then with the limiting probability $\gtrsim (1+\nu^2)^{-1}$,
$Z_n$ is of order $\Bbb E[Z_n]$, i.e. the number of perfect partitions is exponentially large.
\end{theorem}
\noindent {\bf Note.\/} The limiting probability is certainly below 
\[
\tfrac{1}{\nu}\!=\!\lim\Bbb P\Bigl(\sum_{j\in [n]}X_j\!\!\equiv 0 (\text{mod }\nu)\Bigr).
\]
\begin{proof} By Lemma \ref{newlem1}, we have 
\begin{equation}\label{Sec1}
\Bbb E[Z_n^2]=\frac{1}{(2\pi)^{ 2\nu}}\!\!\!\int\limits_{\bold x, \bold x'}\Bbb E^n\bigl[f(\bold x,X)
\overline{f(\bold x',X)}\bigr]\,\,d\bold x d\bold x'.
\end{equation}
Here, with $\phi(y)=\Bbb E[e^{iyX}]$, we have
\begin{multline}\label{Eg=}
\Bbb E\bigl[f(\bold x,X)
\overline{f(\bold x',X)}\bigr]=\Bbb E\biggl[\biggl(\sum_{\a\in [\nu]}\exp(iy_{\a}(\bold x)X)\biggr)
\biggl(\sum_{\be\in [\nu]}\exp(-iy_{\be}(\bold x')X)\biggr)\biggl]\\
=\Bbb E\biggl[\sum_{\a,\be \in [\nu]}\exp\bigl(i(y_{\a}(\bold x)-y_{\be}(\bold x'))X)\biggr]=\sum_{\a,\be\in [\nu]}\phi\bigl(y_{\a}(\bold x)-y_{\be}(\bold x')\bigr).
\end{multline}

Let $\eps\in (0, \nu^{-2})$, and introduce 
\[
D_{\eps}=\Bigl\{\bold x, \bold x'\in [-\pi,\pi]^{\nu}: \min_{\a,\be\in [\nu]}\big|\phi(y_{\a}(\bold x)-y_{\be}(\bold x'))\big|\le 1-\eps\nu^2\Bigr\}.
\]
For $(\bold x, \bold x')\in D_{\eps}$, we have
\[
\sum_{\a,\be\in [\nu]}|\phi\bigl(y_{\a}(\bold x)-y_{\be}(\bold x')\bigr)|\le (1-\eps\nu^2)+\nu^2-1=\nu^2(1-\eps).
\]
So, $\mathcal E[Z_n^2; D_{\eps}]$, the contribution of $D_{\eps}$ to $\Bbb E[Z_n^2]$, is at most $\nu^{2n}(1-\eps)^n$. 

Let $(\bold x,\bold x')\in D_{\eps}^c$, so that $\big|\phi(y_{\a}(\bold x)-y_{\be}(\bold x'))\big|> 1-\eps\nu^2$ for all $\a,\,\be$. If 
$y= y_{\a,\be}:=y_{\a}(\bold x)-y_{\be}(\bold x')\neq k\pi$, $k$ even, then
\begin{multline*}
|\phi(y)|=\bigg|\tfrac{e^{iy}(e^{iMy}-1)}{M(e^{iy}-1)}\bigg|
=\tfrac{\bigl(1-\cos(My)\bigr)^{1/2}}{M(1-\cos y)^{1/2}}
=\bigg|\tfrac{\sin(My/2)}{M\sin (y/2)}\biggr|> 1-\eps\nu^2.\\
\end{multline*}
Let $k_{\a,\be}(\bold x,\bold x')\pi$ be the even multiple of $\pi$ closest to $y_{\a}(\bold x)-y_{\be}(\bold x')$, and set
$z_{\a,\be}(\bold x,\bold x')=y_{\a}(\bold x)-y_{\be}(\bold x')-k_{\a,\be}(\bold x,\bold x')\pi$; then $|z_{\a,\be}(\bold x,\bold x')|\le \pi$, and 
$\phi(y_{\a}(\bold x)-y_{\be}(\bold x'))=\phi(z_{\a,\be}(\bold x,\bold x'))$. Therefore
\begin{multline*}
1-\eps\nu^2\le |\phi(y_{\a}(\bold x)-y_{\be}(\bold x'))|
=|\phi(z_{\a,\be}(\bold x,\bold x'))|\\
\le \frac{1}{M|\sin(z_{\a,\be}(\bold x,\bold x')/2)|}\le\tfrac{\pi}{M|z_{\a,\be}(\bold x,\bold x')|},
\end{multline*}
whence 
\[
|y_{\a}(\bold x)-y_{\be}(\bold x')-k_{\a,\be}(\bold x,\bold x')\pi|= |z_{\a,\be}(\bold x,\bold x')|\le\tfrac{\pi}{M(1-\eps\nu^2)}=O(M^{-1}). 
\]
Since $\sum_{\a}y_{\a}(\bold x)=\sum_{\be}y_{\be}(\bold x')=0$, we obtain then
\begin{align*}
&y_{\a}(\bold x)-\tfrac{\pi}{\nu}k_{\a}(\bold x,\bold x')=O(M^{-1}),\quad k_{\a}(\bold x,\bold x'):=\sum_{\be}k_{\a,\be}(\bold x,\bold x'),\\
&y_{\be}(\bold x')-\tfrac{\pi}{\nu}\kappa_{\be}(\bold x,\bold x')=O(M^{-1}),\quad \kappa_{\be}(\bold x,\bold x'):=\sum_{\a}k_{\a,\be}(\bold x,\bold x').
\end{align*}
Since $\nu$ is fixed and $M\to\infty$, $\tfrac{\pi}{\nu} k_{\a}(\bold x,\bold x')$ ($\tfrac{\pi}{\nu} \kappa_{\be}(\bold x,\bold x')$ resp.) is again an even multiple of $\tfrac{\pi}{\nu}$ closest to $y_{\a}(\bold x)$ ($y_{\be}(\bold x')$ resp.), whence
$k_{\a}(\bold x,\bold x')$ ($\kappa_{\be}(\bold x,\bold x')$ resp.) depends only on $\bold x$ (on $\bold x'$ resp.). 
Combining the last three equations, and again using $\sum_{\a}y_{\a}(\bold x)=0$, $\sum_{\be}y_{\be}(\bold x')=0$, we obtain 
\begin{align*}
&\tfrac{\pi}{\nu}\cdot (k_{\a}(\bold x)-\kappa_{\be}(\bold x'))=\pi k_{\a,\be}(\bold x,\bold x')+O(M^{-1}),\\
&\sum_{\a}k_{\a}(\bold x)=O(M^{-1}),\quad \sum_{\be}\kappa_{\be}(\bold x')=O(M^{-1}),
\end{align*}
implying that for $M$ large
\[
k_{\a,\be}(\bold x,\bold x')=\tfrac{1}{\nu} (k_{\a}(\bold x)-\kappa_{\be}(\bold x')),\quad \sum_{\a}k_{\a}(\bold x)=\sum_{\be}\kappa_{\be}(\bold x')=0.
\]
(In particular, all $\nu^2$ differences $k_{\a}(\bold x)-\kappa_{\be}(\bold x')$ are divisible by $2\nu$.) The first identity above implies that $z(\bold x,\bold x')=M^{-1}\bigl(\xi_{\a}(\bold x)-\eta_{\be}(\bold x')\bigr)$, where
\[
\xi_{\a}(\bold x):=M\bigl(y_{\a}(\bold x)-\tfrac{\pi}{\nu}k_{\a}(\bold x)\bigr),\quad \eta_{\be}(\bold x'):=
M\bigl(y_{\be}(\bold x')-\tfrac{\pi}{\nu}k_{\be}(\bold x')\bigr),
\]
and by the second identity, $\sum_{\a}\xi_{\a}(\bold x)=0$, $\sum_{\be}\eta_{\be}(\bold x')=0$. So,
\[
\phi(z_{\a,\be}(\bold x,\bold x'))=\phi\bigl(M^{-1}(\xi_{\a}(\bold x)-\eta_{\be}(\bold x'))\bigr).
\]
and $|\xi_{\a}(\bold x)-\eta_{\be}(\bold x')|$ is uniformly bounded for $(\bold x,\bold x')\in D_{\eps}^c$.
Now, for each $b>0$, there exists $\la=\la(b)>0$ such that $u^{-1}|\sin u|\in [1-2\la u^2, 1-\la ^2]$ for $u\in (0,b]$.
Since
\[
|\phi(z)|=\biggl|\tfrac{\sin(Mz/2)}{Mz/2}\biggr| \cdot \biggl|\tfrac{z/2}{\sin(z/2)}\biggr|,
\]
we see then that, for $(\bold x,\bold x')\in D_{\eps}^c$ and an absolute constant $\la'>0$,
\begin{multline}\label{|sum|<}
\sum_{\a,\be}|\phi\bigl(y_{\a}(\bold x)-y_{\be}(\bold x')\bigr)|\le \sum_{\a,\be}\Bigl(1-\la' (\xi_{\a}(\bold x)-
\eta_{\be}(\bold x'))^2\Bigr)\\
\le \nu^2\exp\biggl(-\tfrac{\la'}{\nu^2}\sum_{\a,\be}(\xi_{\a}(\bold x)-\eta_{\be}(\bold x'))^2\biggr)\\
=\nu^2\exp\Bigl(-\tfrac{\la'}{\nu^2}\bigl(\|\boldsymbol\xi(\bold x)\|^2+\|\boldsymbol \eta(\bold x')\|^2\bigr)\Bigr),
\end{multline}
since $\sum_{\a}\xi_{\a}(\bold x)=\sum_{\be}\eta_{\be}(\bold x')=0$. 

Let $\mathcal E[Z_n^2; D_{\eps}^c]$ be the contribution to $\Bbb E[Z_n^2]$ from $(\bold x,\bold x')\in D_{\eps}^c$  such that
\[
\max\bigl\{\|\boldsymbol\xi(\bold x)\|,\,\|\boldsymbol \eta(\bold x')\|\bigr\}\ge \tfrac{\log n}{n^{1/2}}.
\]
Using \eqref{|sum|<}, similarly to the part {\bf (ii)\/} of the proof of Lemma \ref{lem2.55}, we obtain 
\begin{equation*}
\mathcal E[Z_n^2; D_{\eps}^c]\le c_1(\nu)\tfrac{\nu^{2n}}{\la' M^{2(\nu-1)}}\exp\Bigl(-\tfrac{\la'\log^2 n}{\nu^2}\Bigr).
\end{equation*}
Consequently
\begin{equation}\label{per}
\begin{aligned}
\mathcal E[Z_n^2; D_{\eps}]+\mathcal E[Z_n^2; D_{\eps}^c] &=O\biggl(\nu^{2n}(1-\eps)^n+\tfrac{\nu^{2n}}{\la' M^{2(\nu-1)}}\exp\Bigl(-\tfrac{\la'\log^2 n}{\nu^2}\Bigr)\biggr)\\
&=O\biggl(\tfrac{\nu^{2n}}{\la' M^{2(\nu-1)}}\exp\Bigl(-\tfrac{\la'\log^2 n}{\nu^2}\Bigr)\biggr),
\end{aligned}
\end{equation}
provided that 
\begin{equation}\label{prov}
\lim\tfrac{n}{\log M}>\frac{2(\nu-1)}{\log\tfrac{1}{1-\nu^{-2}}}.
\end{equation}
It remains to sharply estimate $\mathcal E[Z_n^2;\mathcal C]$, the contribution to $\Bbb E[Z_n^2]$ from 
\[
\mathcal C^*:=\Bigl\{(\bold x,\bold x')\in D_{\eps}^c:\max\bigl\{\|\boldsymbol\xi(\bold x)\|,\,\|\boldsymbol \eta(\bold x')\|\bigr\}\le \tfrac{\log n}{n^{1/2}}\Bigr\}.
\]
Analogously to $\Bbb E[Z_n]$, (see the part {\bf (III)\/} of the proof of Theorem \ref{newthm1}), $\mathcal C$ is the disjoint union of $\mathcal C^*_{\bold k,\boldsymbol\kappa}$,
over all pairs $\{\bold k,\boldsymbol\kappa\}$ of $\nu$-long even tuples $\bold k=\{k_{\a}\}$, $\boldsymbol\kappa=\{\kappa_{\be}\}$, with $\sum_{\a}k_{\a}=\sum_{\be}\kappa_{\be}=0$, $k_{\a}-\kappa_{\be}\equiv 0(\text{mod }2\nu)$, and
\begin{align*}
\mathcal C^*_{\bold k,\boldsymbol\kappa}=\Bigl\{(\bold x,\bold x')\in D_{\eps}^c:& \max\bigl\{\|\boldsymbol\xi(\bold x)\|,
\|\boldsymbol \eta(\bold x')\|\bigr\}\\
&\le \tfrac{\log n}{n^{1/2}},\, k_{\a}(\bold x)\!=k_{\a},\, \kappa_{\be}(\bold x')\!=\!\kappa_{\be}\Bigr\}.
\end{align*}
Geometrically, $\mathcal C^*_{\bold k,\boldsymbol\kappa}$ is the Cartesian product of two thin $\nu$-dimensional cylinders, enclosing respectively the line $\mathcal L^*_{\bold k}$ and the line $\mathcal L^*_{\boldsymbol\kappa}$,
given by the parametric equations 
\begin{equation}\label{par}
x_{\a}(t)=\nu^{-1}\bigl(t+\tfrac{\pi}{\nu}k_{\a}\bigr),\,\a\in [\nu],\quad x'_{\be}(t')=\nu^{-1}\bigl(t'+\tfrac{\pi}{\nu}\kappa_{\be}\bigr),\,\be\in [\nu].
\end{equation}
As in the case of $\Bbb E[Z_n]$, the dominant contributors to $\Bbb E[Z_n^2]$ are the pairs $\{\mathcal L^*_{\bold k},
\mathcal L^*_{\boldsymbol\kappa}\}$ that contain interior points of their respective cubes $[-\pi,\pi]^{\nu}$. However this time we need to consider only the pairs for which the differences $k_{\a}-\kappa_{\be}$ are all divisible by $2\nu$. Each of the cross-sections of the two cylinders formed by planes $\sum_{\a}x_{\a}=t$, 
and $\sum_{\be}x'_{\be}=t'$ respectively, is the $(\nu-1)$-dimensional sphere of radius $r_n:=(M\nu)^{-1}n^{-1/2}\log n$, coming from
\begin{align*}
\|\bold x-\bold x(t)\|^2&=\tfrac{1}{(M\nu)^2}\sum_{\a}\xi_{\a}^2(\bold x)\le \tfrac{\log^2n}{(M\nu)^2n},\\
\|\bold x'-\bold x'(t')\|^2&=\tfrac{1}{(M\nu)^2}\sum_{\be}\eta_{\be}^2(\bold x')\le \tfrac{\log^2n}{(M\nu)^2n}.
\end{align*}
For $(\bold x,\bold x')\in \mathcal C^*_{\bold k,\boldsymbol\kappa}$, we have 
\begin{multline*}
\phi(y_{\a}(\bold x)-y_{\be}(\bold x'))=\phi(z_{\a,\be}(\bold x,\bold x'))=\phi\bigl(M^{-1}(\xi_{\a}(\bold x)-\eta_{\be}(\bold x'))\bigr)\\
=1+i\tfrac{\Bbb E[X]}{M}(\xi_{\a}(\bold x)-\eta_{\be}(\bold x'))-\tfrac{\Bbb E[X^2]}{2M^2}(\xi_{\a}(\bold x)-\eta_{\be}(\bold x'))^2+O\bigl(|\xi_{\a}(\bold x)-\eta_{\be}(\bold x')|^3\bigr).
\end{multline*}
Summing this equation over $\a,\be \in [\nu]$, and using $\sum_{\a}\xi_{\a}(\bold x)=0$, $\sum_{\be}\eta_{\be}(\bold x')=0$, we transform \eqref{Eg=} into
\begin{multline*}
\Bbb E\bigl[f(\bold x,X)\overline{f(\bold x',X)}\bigr]=\nu^2-\tfrac{\Bbb E[X^2]}{2M^2}\bigl(\|\boldsymbol\xi(\bold x)\|^2+
\|\boldsymbol \eta(\bold x')\|^2\bigr)\\
\quad+O\bigl(\|\boldsymbol\xi(\bold x)\|^3+
\|\boldsymbol \eta(\bold x')\|^3\bigr)\\
=\nu^2\Bigl(1-\tfrac{M^2c_M}{2}\bigl(\|\bold x-\bold x(t)\|^2+\|\bold x'-\bold x'(t')\|^2\bigr)\\
+O\bigl(M^3\bigl(\|\bold x-\bold x(t)\|^3+\|\bold x'-\bold x'(t')\|^3\bigr)\Bigr)\\
=\nu^2\exp\Bigl(-\tfrac{M^2c_M}{2}\bigl(\|\bold x-\bold x(t)\|^2+\|\bold x'-\bold x'(t')\|^2\bigr)+O(n^{-3/2}\log^3 n)\Bigr).
\end{multline*}
So
\begin{multline*}
\Bbb E^n\bigl[f(\bold x,X)\overline{f(\bold x',X)}\bigr]\\
=\nu^{2n}\exp\Bigl(-\tfrac{M^2nc_M}{2}\bigl(\|\bold x-\bold x(t)\|^2+\|\bold x'-\bold x'(t')\|^2\bigr)+O(n^{-1/2}\log^3 n)\Bigr).
\end{multline*}
So $\mathcal E\bigl[Z_n^2;\mathcal C^*_{\bold k,\boldsymbol{\kappa}}\bigr]$, the contribution of $\mathcal C^*_{\bold k,\boldsymbol{\kappa}}$ to $\Bbb E[Z_n^2]$, is given by
\[
\mathcal E\bigl[Z_n^2;\mathcal C^*_{\bold k,\boldsymbol{\kappa}}\bigr]=(1+o(1))\tfrac{\nu^{2n}L^*(\bold k)L^*(\boldsymbol\kappa)}
{(2\pi)^{2\nu}(M\nu)^{2(\nu-1)}}\biggl(\int_{\rho\ge 0}e^{-\tfrac{c_Mn}{2\nu}\rho^2}\,dV(\rho)\biggr)^2,
\]
$V(\rho)$ being the volume of $(\nu-1)$-dimensional sphere of radius $\rho$. Here $L^*(\bold k)$ and 
$L^*(\boldsymbol\kappa)$ are the lengths of segments of $\mathcal L^*(\bold k)$ and $\mathcal L^*(\boldsymbol\kappa)$ within the respective cubes $[-\pi,\pi]^{\nu}$. Summing over all admissible pairs $\{\bold k,\boldsymbol\kappa\}$, we have
\[
\mathcal E\bigl[Z_n^2;\mathcal C^*\bigr]\!=\!(1+o(1))\biggl(\!\tfrac{\nu^n}{(2\pi)^{\nu}(M\nu)^{\nu-1}}\int_{\rho\ge 0}e^{-\tfrac{c_Mn}{2\nu}\rho^2}\,dV(\rho)\!\biggr)^2 \sum_{\{\bold k,\boldsymbol\kappa\}}\!L^*(\bold k) L^*(\boldsymbol\kappa).
\]
We upper bound the last product by $(L^*)^2$, $L^*:=\sum_{\bold k}L^*(\bold k)$, thus replacing the condition that 
all $k_{\a}-\kappa_{\be}$  are divisible by $2\nu$ with the weaker condition that {\it marginally\/} all the differences $k_{\a}-k_{\a'}$ and all the differences $\kappa_{\be}-\kappa_{\be'}$ are divisible by $2\nu$. 

By \eqref{par}, for the interior segment of $\mathcal L^*(\bold k)$ we have 
\[
-\pi\bigl(\nu+\tfrac{\min k_{\a}}{\nu}\bigr)<t< \pi\bigl(\nu-\tfrac{\max k_{\a}}{\nu}\bigr).
\]
So,
\[
L^{*}(\bold k)=\pi d^*(\bold k)\nu^{-1/2},\quad d^*(\bold k):=2\nu-\tfrac{r}{\nu}, \,\,r:=\max k_{\a}-\min k_{\a},
\]
and $d^*(\bold k)>0$ if either $k_{\a}\equiv 0$, or $\tfrac{r}{2\nu}\in [\nu-1]$. Consequently
\begin{align*}
\sum_{\bold k}d^*(\bold k)&=2\nu(1+\mathcal M_1^*)-\tfrac{1}{\nu}\mathcal M_2^*,\\
\mathcal M_1^*&=\sum_{\tfrac{r}{2\nu}\in [\nu-1]}\mathcal N^*(r),\quad \mathcal M_2^*=\sum_{\tfrac{r}{2\nu}\in [\nu-1]}r\mathcal N^*(r),
\end{align*}
Here, following closely the line enumeration in the part {\bf (I)\/} of the proof of Theorem \ref{newthm1}, we have
\begin{align*}
&\quad\quad\quad\mathcal N^*(r)=\sum_{\text{even }a\le 0}[\zeta^{-a\nu}]\mathcal F^*_r(\zeta),\quad
\tfrac{r}{2\nu}\in [\nu-1], \\
&\mathcal F^*_r(\zeta)=
\bigl(\tfrac{1-\zeta^{r+2\nu}}{1-\zeta^{2\nu}}\bigr)^{\nu}-\bigl(\tfrac{1-\zeta^r}{1-\zeta^{2\nu}}\bigr)^{\nu}
-\bigl(\tfrac{\zeta^{2\nu}-\zeta^{r+2\nu}}{1-\zeta^{2\nu}}\bigr)^{\nu}+\bigl(\tfrac{\zeta^{2\nu}-\zeta^r}{1-\zeta^{2\nu}}\bigr)^{\nu}.
\end{align*}
So
\begin{align*}
\mathcal M_1^*&=\sum_{\text{even }a\le 0}[\zeta^{-a\nu}]\sum_{\tfrac{r}{2\nu}\in [\nu-1]}\mathcal F^*_r(\zeta)\\
&=\sum_{\text{even }a\le 0}[\zeta^{-a\nu}]\Bigl(\bigl(\tfrac{1-\zeta^{2\nu^2}}{1-\zeta^{2\nu}}\bigr)^{\nu}-\bigl(\tfrac{\zeta^{2\nu}-\zeta^{2\nu^2}}{1-\zeta^{2\nu}}\bigr)^{\nu}-1\Bigr)\\
&=\sum_{\a\ge 0}[t^{\a}]\Bigl[\bigl(\tfrac{1-t^{\nu}}{1-t}\bigr)^{\nu}-\bigl(\tfrac{t-t^{\nu}}{1-t}\bigr)^{\nu}-1\Bigr]\\
&=\lim_{t\to 1}\Bigl[\bigl(\tfrac{1-t^{\nu}}{1-t}\bigr)^{\nu}-\bigl(\tfrac{t-t^{\nu}}{1-t}\bigr)^{\nu}-1\Bigr]
= \nu^{\nu}-(\nu-1)^{\nu}-1.
\end{align*}
Next, with a bit of telescoping again,
\begin{multline*}
\sum_{\tfrac{r}{2\nu}\in [\nu-1]}r\mathcal F^*_r(\zeta)=2\nu(\nu-1)\bigl(\tfrac{1-\zeta^{2\nu^2}}{1-\zeta^{2\nu}}\bigr)^{\nu}\\
-(2\nu(\nu-1)\zeta^{2\nu^2}+2\nu)\bigl(\tfrac{1-\zeta^{2\nu(\nu-1)}}{1-\zeta^{2\nu}}\bigr)^{\nu}\\
+2\nu(\zeta^{2\nu^2}-1)\sum_{j=1}^{\nu-2}\bigl(\tfrac{1-\zeta^{2\nu j}}{1-\zeta^{2\nu}}\bigr)^{\nu}=:S^*(\zeta^{2\nu}).
\end{multline*}
Consequently,
\begin{align*}
\mathcal M_2^*&=\sum_{\text{even }a\le 0}[\zeta^{-a\nu}]\sum_{\tfrac{r}{2\nu}\in [\nu-1]}r\mathcal F^*_r(\zeta)
=\sum_{\a\ge 0}[t^{\a}] S^*(t)\\
&=\lim_{t\to 1}S^*(t)=2\nu(\nu-1)\nu^{\nu}-2\nu^2(\nu-1)^{\nu}.
\end{align*}
Therefore
\begin{align*}
\sum_{\bold k}d^*(\bold k)&=2\nu(1+\mathcal M_1^*)-\tfrac{1}{\nu}\mathcal M_2^*\\
&=2\nu(\nu^{\nu}-(\nu-1)^{\nu})-2(\nu-1)\nu^{\nu}+2\nu(\nu-1)^{\nu}=2\nu^{\nu},
\end{align*}
so that $L^*(\bold k)=2\pi\nu^{\nu-1/2}$. So,
\begin{align*}
\mathcal E\bigl[Z_n^2;\mathcal C^*\bigr]\!&\le\!(1+o(1))\biggl(\!\tfrac{\nu^n \cdot2\pi\nu^{\nu-1/2}}{(2\pi)^{\nu}(M\nu)^{\nu-1}}\int_{\rho\ge 0}e^{-\tfrac{c_Mn}{2\nu}\rho^2}\,dV(\rho)\!\biggr)^2\\
&=(1+o(1))\left(\tfrac{\nu^n}{M^{\nu}-1}\cdot\frac{\nu^{\nu-1/2}}{(2\pi\nu c_{M}n)^{\tfrac{\nu-1}{2}}}\right)^2.
\end{align*}
In combination with \eqref{per}, and \eqref{clear1}, this leads to 
\begin{align*}
\Bbb E[Z_n^2]&\le (1+o(1))\left(\tfrac{\nu^n}{M^{\nu}-1}\cdot\frac{\nu^{\nu-1/2}}{(2\pi\nu c_{M}n)^{\tfrac{\nu-1}{2}}}\right)^2\\
&=(1+o(1)) \nu^2 \Bbb E^2[Z_n],
\end{align*}
provided that $\lim\tfrac{n}{\log M}>\tfrac{2(\nu-1)}{\log[(1-2\nu^{-2})^{-1}]}$.
By Cantellli's inequality (Billingsley \cite{Bil}), 
\begin{multline*}
\liminf\Bbb P\bigl(Z_n\ge \delta\Bbb E[Z_n]\bigr)\ge \lim \tfrac{(1-\delta)^2}{(1-\delta)^2+
\tfrac{\Bbb E[Z^2_n]}{\Bbb E^2[Z_n]}}\ge \tfrac{(1-\delta)^2}{(1-\delta)^2+\nu^2},\quad\forall\,\delta\in (0,1).\\
\end{multline*}
\end{proof}
\noindent {\bf Acknowledgment.\/} It is my pleasure to thank the participants of combinatorial probability seminar at Ohio State University for
encouraging discussion of this study. I owe a debt of gratitude to Huseyin Acan for alerting me to a possibility of an oversight in the
preliminary draft, and to Hoi Nguyen for pinpointing it with surgical precision. I am grateful to George Varghese for drawing my attention in Summer 2021 to the perfect multiway partitions. The thought-provoking paper that George and  his son Tim wrote on the subject slowly but surely pulled me in.

\end{document}